\documentclass[12pt]{amsart}%
\usepackage{amsmath}
\usepackage{amsfonts}
\usepackage{amssymb}
\usepackage{amstext}
\usepackage{graphicx}%
\providecommand{\U}[1]{\protect \rule{.1in}{.1in}}
\providecommand{\U}[1]{\protect \rule{.1in}{.1in}}
\providecommand{\U}[1]{\protect \rule{.1in}{.1in}}
\newtheorem{theorem}{Theorem}[section]
\newtheorem{lemma}{Lemma}[section]
\newtheorem{proposition}{Proposition}[section]

\numberwithin{equation}{section}
\newtheorem {conjecture}{Conjecture}

\theoremstyle{remark}
\newtheorem{remark}{Remark}[section]
\numberwithin{equation}{section}
\begin{document}
\title[sharp dimension estimate of CR holomorphic functions]{On the sharp dimension estimate of CR holomorphic functions in Sasakian Manifolds}
\author{$^{\ast}$Shu-Cheng Chang}
\address{Department of Mathematics and Taida Institute for Mathematical Sciences
(TIMS), National Taiwan University, Taipei 10617, Taiwan, R.O.C.}
\email{scchang@math.ntu.edu.tw }
\author{$^{\dag}$Yingbo Han}
\address{{School of Mathematics and Statistics, Xinyang Normal University}\\
Xinyang,464000, Henan, P.R. China}
\email{{yingbohan@163.com}}
\author{$^{\ast \ast}$Chien Lin}
\address{Department of Mathematics, National Tsing Hua University, Hsinchu 30013,
Taiwan, R.O.C}
\email{r97221009@ntu.edu.tw}
\thanks{$^{\ast}$Shu-Cheng Chang and $^{\ast \ast}$Chien Lin are partially supported in
part by the MOST of Taiwan.}
\thanks{$^{\dag}$Yingbo Han is partially supported by an NSFC grant No. 11201400 and
Nanhu Scholars Program for Young Scholars of {Xinyang Normal University}.}
\subjclass{Primary 32V05, 32V20; Secondary 53C56.}
\keywords{CR heat kernel, CR Yau's uniformization conjecture, CR holomorphic function,
Sharp dimension estimate.}

\begin{abstract}
This is the very first paper to focus on the CR analogue of Yau's
uniformization conjecture in a complete noncompact pseudohermitian
$(2n+1)$-manifold of vanishing torsion (i.e. Sasakian manifold) which is an
odd dimensional counterpart of K\"{a}hler geometry. In this paper, we mainly
deal with the problem of the sharp dimension estimate of CR holomorphic
functions in a complete noncompact pseudohermitian manifold of vanishing
torsion with nonnegative pseudohermitian bisectional curvature.

\end{abstract}
\maketitle

\section{Introduction}

S.-Y. Cheng and S.-T. Yau( \cite{y1}, \cite{cy}) derived the well-known
gradient estimate for positive harmonic functions and obtained the classical
Liouville theorem, which states that any bounded harmonic function is constant
in complete noncompact $m$-dimensional Riemannian manifolds with nonnegative
Ricci curvature. Yau conjectured that the dimension $h^{d}(M)$ of the space
$\mathcal{H}^{d}(M)$ consisting of harmonic functions of polynomial growth of
degree at most $d,$ is finite for each positive integer $d$ and satisfies the
estimate:
\[
h^{d}(M)\leq h^{d}(\mathbf{R}^{m}).
\]
Colding and Minicozzi (\cite{cm}) affirmatively answered the first question
and proved that
\[
h^{d}(M)\leq C_{0}d^{m-1}%
\]
for manifolds of nonnegative Ricci curvature. Later, Li (\cite{li2}) produced
a shorter proof requiring only the manifold to satisfy the volume doubling
property and the mean value inequality. For the latter question, the sharp
upper bound estimate is still missing except for the special cases $m=2$ or
$d=1$ obtained by Li-Tam\ (\cite{lt}, \cite{lt1}) and Kasue (\cite{k}), and
the rigidity part is only known for the special case $d=1$ obtained by Li
(\cite{li1}) and Cheeger-Colding-Minicozzi (\cite{ccm}).

In \cite{ckt}, by modifying the arguments of \cite{y1}, \cite{cy}, and
\cite{ckl}, Chang, Kuo and Tie derived a sub-gradient estimate for positive
pseudoharmonic functions in a complete noncompact pseudohermitian manifold
$(M,J,\theta)$. This sub-gradient estimate can serve as the CR version of
Yau's gradient estimate. As an application of the sub-gradient estimate, the
CR analogue of Liouville-type theorem holds for positive pseudoharmonic
functions. In the recent paper (\cite{ccht}), we study the CR analogue of
Yau's conjecture on the space $\mathcal{H}^{d}(M)$ consisting of all
pseudoharmonic functions of polynomial growth of degree at most $d$ in a
complete noncompact pseudohermitian manifold. We showed that the first part of
CR Yau's conjecture holds for pseudoharmonic functions of polynomial growth.

In K\"{a}hler geometry, Yau (\cite{scy}) proposed a variety of
uniformization-type problems on complete K\"{a}hler manifolds with nonnegative
holomorphic bisectional curvature. The first conjecture is that if $M$ is a
complete noncompact $n$-dimensional K\"{a}hler manifold with nonnegative
holomorphic bisectional curvature, then
\[
\dim_{%
\mathbb{C}
}\left(  \mathcal{O}_{d}\left(  M^{n}\right)  \right)  \leq \dim_{%
\mathbb{C}
}\left(  \mathcal{O}_{d}\left(
\mathbb{C}
^{n}\right)  \right)  .
\]
The equality holds if and only if $M$ is isometrically biholomorphic to $%
\mathbb{C}
^{n}$. Here $\mathcal{O}_{d}\left(  M^{n}\right)  $ denotes the family of all
holomorphic functions on a complete $n$-dimensional K\"{a}hler manifold $M$ of
polynomial growth of degree at most $d$. In \cite{n}, Ni established the
validity of this conjecture by deriving the monotonicity formula for the heat
equation under the extra assumption that $M$ has maximal volume growth
\[
\lim_{r\rightarrow+\infty}\frac{Vol\left(  B_{p}\left(  r\right)  \right)
}{r^{2n}}\geq c
\]
for a fixed point $p$ and a positive constant $c$. Later, in \cite{cfyz},
Chen, Fu, Yin, and Zhu improved Ni's result without the assumption of maximal
volume growth. One should refer to \cite{liu1} for more general results recently.

The second conjecture is that if $M$ is a complete noncompact $n$-dimensional
K\"{a}hler manifold with nonnegative holomorphic bisectional curvature, then
the ring $O_{P}\left(  M\right)  $ of all holomorphic functions of polynomial
growth is finitely generated. This one was solved completely by G. Liu
(\cite{liu2}) quite recently. He mainly deployed four techniques to attack
this conjecture via Cheeger-Colding-Tian's theory (\cite{chco1}, \cite{chco2},
\cite{cct}), methods of heat flow developed by Ni and Tam (\cite{n},
\cite{nt1}, \cite{nt4}), Hormander $L^{2}$-estimate of $\overline{\partial}$
(\cite{de}) and three circle theorem (\cite{liu1}).

The third uniformization conjecture is that if $M$ is a complete noncompact
$n$-dimensional K\"{a}hler manifold with positive holomorphic bisectional
curvature, then $M$ is biholomorphic to the standard $n$-dimensional complex
space $%
\mathbb{C}
^{n}$. The first giant progress pertaining to the third conjecture could be
attributed to Mok, Siu and Yau. In their papers (\cite{msy} and \cite{m1}),
they showed that, under the assumptions of the maximal volume growth and the
scalar curvature $R\left(  x\right)  $ decays as
\[
0\leq r\left(  x\right)  \leq \frac{C}{\left(  1+d\left(  x,x_{0}\right)
\right)  ^{2+\epsilon}}%
\]
for some positive constant $C$ and any arbitrarily small positive number
$\epsilon$, a complete noncompact $n$-dimensional K\"{a}hler manifold $M$ with
nonnegative holomorphic bisectional curvature is isometrically biholomorphic
to $%
\mathbb{C}
^{n}$. A Riemannian version was proved in \cite{gw2} shortly afterwards. Since
then there are several further works aiming to prove the optimal result and
reader is referred to \cite{m2}, \cite{ctz}, \cite{cz}, \cite{n2}, \cite{nt1},
\cite{nt2} and \cite{nst}. For example, A. Chau and L. F. Tam (\cite{ct})
proved that a complete noncompact K\"{a}hler manifold with bounded nonnegative
bisectional curvature and maximal volume growth is biholomorphic to $%
\mathbb{C}
^{n}$. Recently, G. Liu (\cite{liu3}) confirmed Yau's uniformization
conjecture when $M$ has maximal volume growth.

This is the very first paper to focus on the CR analogue of Yau's
uniformization conjecture in a complete noncompact pseudohermitian
$(2n+1)$-manifold of vanishing torsion (i.e. Sasakian manifold) which is an
odd dimensional counterpart of K\"{a}hler geometry. That is

\begin{conjecture}
Let $M$ be a complete noncompact pseudohermitian $(2n+1)$-manifold of
vanishing torsion with positive pseudohermitian bisectional curvature. Then
$M$ is CR biholomorphic to the standard Heisenberg group $\mathbf{H}_{n}$ $=$
$%
\mathbb{C}
^{n}\times%
\mathbb{R}
$.
\end{conjecture}

In this paper, it is very natural to concerned the CR analogue of the first
conjecture in a complete noncompact pseudohermitian $(2n+1)$-manifold of
vanishing torsion with nonnegative pseudohermitian bisectional curvature. A
smooth complex-valued function on a pseudohermitian $\left(  2n+1\right)
$-manifold $\left(  M,J,\theta \right)  $ is called CR-holomorphic if%
\[
\overline{\partial}_{b}f=0.
\]
We refer to the next section for the detailed notations. For any fixed point
$x\in M$, a CR-holomorphic function $f$ is called to be of polynomial growth
if there are a nonnegative number $d$ and a positive constant $C=C\left(
x,d,f\right)  $, depending on $x$, $d$ and $f$, such that%
\[
\left \vert f\left(  y\right)  \right \vert \leq C\left(  1+d_{cc}\left(
x,y\right)  \right)  ^{d}%
\]
for all $y\in M$, where $d_{cc}\left(  x,y\right)  $ denotes the
Carnot-Caratheodory distance between $x$ and $y.$ In the following sections,
we sometimes would use the notation $r\left(  x,y\right)  $ for the
Carnot-Caratheodory distance. Furthermore, we could define the degree of a
CR-holomorphic function $f$ of polynomial growth by%
\[
\deg \left(  f\right)  =\inf \left \{  d\geq0\left \vert
\begin{array}
[c]{c}%
\left \vert f\left(  y\right)  \right \vert \leq C\left(  1+d_{cc}\left(
x,y\right)  \right)  ^{d}\text{ \ }\forall \text{ }y\in M,\\
for\text{ }some\text{ }d\geq0\text{ }and\text{ }C=C\left(  x,d,f\right)
\end{array}
\right.  \right \}
\]
as well as the aforementioned holomorphic case. In fact, the definition above
is independent of the choice of the point $x\in M$. Finally we denote
$\mathcal{O}_{d}^{CR}(M)$ the family of all CR-holomorphic functions $f$ of
polynomial growth of degree at most $d$ with $Tf(x)=f_{0}(x)=0:$
\[
\mathcal{O}_{d}^{CR}(M)=\{f(x)\ |\overline{\partial}_{b}f(x)=0,\ f_{0}%
(x)=0\text{ and }|f(x)|\leq C\left(  1+d_{cc}\left(  x,y\right)  \right)
^{d}\text{ for some constant }C\  \}.
\]

Now we explain the extra condition $Tf=0$ in the definition of $\mathcal{O}%
_{d}^{CR}\left(  M\right)  $ in a Sasakian manifold. Follow the notion as in
\cite{fow} : \ Let $\{U_{\alpha}\}_{\alpha \in A}$ be an open covering of the
Sasakian manifold $(M^{2n+1},g)$ and $\pi_{\alpha}:$ $U_{\alpha}\rightarrow
V_{\alpha}\subset%
\mathbb{C}
^{n}$ submersion such that $\pi_{\alpha}\circ \pi_{\beta}^{-1}:\pi_{\beta
}(U_{\alpha}\cap U_{\beta})\rightarrow \pi_{\alpha}(U_{\alpha}\cap U_{\beta})$
is biholomorphic. On each $V_{\alpha},$ there is a canonical isomorphism
\[
d\pi_{\alpha}:D_{p}\rightarrow T_{\pi_{\alpha}(p)}V_{\alpha}%
\]
for any $p\in U_{\alpha},$ where $D=\ker \theta \subset TM.$ Since $T$ generates
isometries, the restriction of the Sasakian metric $g$ to $D$ gives a
well-defined Hermitian metric $g_{\alpha}^{T}$ on $V_{\alpha}.$ This Hermitian
metric in fact is K\"{a}hler. More precisely, let $z^{1},z^{2},\cdot \cdot
\cdot,z^{n}$ be the local holomorphic coordinates on $V_{\alpha}$. We pull
back these to $U_{\alpha}$ and still write the same. Let $x$ be the coordinate
along the leaves with $T=\frac{\partial}{\partial x}.$ Then we have the local
coordinate $\{x,z^{1},z^{2},\cdot \cdot \cdot,z^{n}\}$ on $U_{\alpha}\ $and
$(D\otimes%
\mathbb{C}
)^{1,0}$ is spanned by the form
\[
Z_{\alpha}=(\frac{\partial}{\partial z^{\alpha}}-\theta(\frac{\partial
}{\partial z^{\alpha}})T),\  \  \  \alpha=1,2,\cdot \cdot \cdot,n.
\]
Since $i(T)d\theta=0,$
\[
d\theta(Z_{\alpha},\overline{Z_{\beta}})=d\theta(\frac{\partial}{\partial
z^{\alpha}},\frac{\partial}{\partial \overline{z}^{\beta}}).
\]
Then the K\"{a}hler $2$-form $\omega_{\alpha}^{T}$ of the Hermitian metric
$g_{\alpha}^{T}$ on $V_{\alpha},$ which is the same as the restriction of the
Levi form $\frac{1}{2}d\theta$ to $\widetilde{D_{\alpha}^{n}}$, the slice
$\{x=$ \textrm{constant}$\}$ in $U_{\alpha},$ is closed. The collection of
K\"{a}hler metrics $\{g_{\alpha}^{T}\}$ on $\{V_{\alpha}\}$ is so-called a
transverse K\"{a}hler metric. We often refer to $\frac{1}{2}d\theta$ as the
K\"{a}hler form of the transverse K\"{a}hler metric $g^{T}$in the leaf space
$\widetilde{D^{n}}.$ As an example, $\left \{  Z_{\alpha}=\frac{\partial
}{\partial z^{\alpha}}+i\overline{z}^{\alpha}\frac{\partial}{\partial
t}\right \}  _{\alpha=1}^{n}$ is exactly a local frame on in the $(2n+1)$%
-dimensional Heisenberg group $\mathbf{H}_{n}$ $=$ $%
\mathbb{C}
^{n}\times%
\mathbb{R}
$. Here
\[
\theta=dt+i%
{\displaystyle \sum \limits_{\alpha \in I_{n}}}
\left(  z^{\alpha}d\overline{z}^{\alpha}-\overline{z}^{\alpha}dz^{\alpha
}\right)
\]
is a pseudohermitian contact structure on $\mathbf{H}_{n}$ and $T=\frac
{\partial}{\partial t}$. In this case, $\widetilde{D^{n}}=%
\mathbb{C}
^{n}$ and then
\begin{equation}
\dim_{%
\mathbb{C}
}\mathcal{O}_{d}(%
\mathbb{C}
^{n})=\dim_{%
\mathbb{C}
}\left(  \mathcal{O}_{d}^{CR}\left(  \mathbf{H}_{n}\right)  \right)  .
\label{r1}%
\end{equation}

Now, for a nontrivial function $f\in$ $\mathcal{O}_{d}^{CR}\left(  M\right)  $
with $f\left(  x\right)  =0$, we would define the vanishing order of $f$ at
$x\in M$ by%
\[
ord_{x}\left(  f\right)  =\max \left \{  m\in%
\mathbb{N}
\text{ }|\ D^{\alpha}f=0\text{,\  \  \  \ }\forall \text{ }\left \vert
\alpha \right \vert <m\  \right \}  ,
\]
where $D^{\alpha}=%
{\displaystyle \prod \limits_{j\in I_{n}}}
Z_{j}^{\alpha_{j}}$ with $\alpha=\left(  \alpha_{1},\alpha_{2,}...,\alpha
_{n}\right)  $.

Now we state our main theorem in this paper.

\begin{theorem}
\label{TA} If $(M,J,\theta)$ is a complete noncompact pseudohermitian
$(2n+1)$-manifold of vanishing torsion with nonnegative pseudohermitian
bisectional curvature, then%
\begin{equation}
\dim_{%
\mathbb{C}
}\left(  \mathcal{O}_{d}^{CR}\left(  M\right)  \right)  \leq \dim_{%
\mathbb{C}
}\left(  \mathcal{O}_{d}^{CR}\left(  \mathbf{H}_{n}\right)  \right)
\label{339}%
\end{equation}
for each $d\in%
\mathbb{N}
$ and $\mathbf{H}_{n}$ = $%
\mathbb{C}
^{n}\times%
\mathbb{R}
$ is the $(2n+1)$-dimensional Heisenberg group.
\end{theorem}

\begin{remark}
1. In order to prove Theorem \ref{TA}, it follows from (\ref{r1}) that it
suffices to show that
\begin{equation}
\dim_{%
\mathbb{C}
}\left(  \mathcal{O}_{d}^{CR}\left(  M\right)  \right)  \leq \dim_{%
\mathbb{C}
}\mathcal{O}_{d}\left(
\mathbb{C}
^{n}\right)  \label{339A}%
\end{equation}
in a complete noncompact pseudohermitian $(2n+1)$-manifold of vanishing
torsion with nonnegative pseudohermitian bisectional curvature.

2. \ It is interesting to know whether $M$ is CR equivalent to the Heisenberg
group $\mathbf{H}_{n}$ if the equality holds
\[
\dim_{%
\mathbb{C}
}\left(  \mathcal{O}_{d}^{CR}\left(  M\right)  \right)  =\dim_{%
\mathbb{C}
}\left(  \mathcal{O}_{d}^{CR}\left(  \mathbf{H}_{n}\right)  \right)  .
\]
\bigskip
\end{remark}

The method here is inspired mainly from \cite{n} and \cite{cfyz} which is
organized as follows. In Section $2$, we introduce some basic notions about
pseudohermitian manifolds and the necessary results for this paper. In Section
$3$, we show the existence of solutions to the CR heat equation with the
initial condition under some appropriate assumptions. In Section $4$, we prove
the CR analogue of the rough dimension estimate. In Section $5$, we show the
CR sharp dimension estimate. In Appendix, we would give the complete proof of
the $L^{p}$-submean value inequality which is a key step in showing the
existence of solutions to the CR heat equation in Section $3$.

\section{Preliminaries}

We introduce some basic materials about a pseudohermitian manifold (see
\cite{l} and \cite{dt} for more details). Let $(M,\xi)=\left(  M,J,\theta
\right)  $ be a $(2n+1)$-dimensional, orientable, contact manifold with the
contact structure $\xi$. A CR structure compatible with $\xi$ is an
endomorphism $J:\xi \rightarrow \xi$ such that $J^{2}=-1$. We also assume that
$J$ satisfies the integrability condition: If $X$ and $Y$ are in $\xi$, then
so are $[JX,Y]+[X,JY]$ and $J([JX,Y]+[X,JY])=[JX,JY]-[X,Y]$.

Let $\left \{  T,Z_{\alpha},Z_{\bar{\alpha}}\right \}  _{\alpha \in I_{n}}$ be a
frame of $TM\otimes \mathbb{C}$, where $Z_{\alpha}$ is any local frame of
$T_{1,0}M,\ Z_{\bar{\alpha}}=\overline{Z_{\alpha}}\in T_{0,1}M$ ,and $T$ is
the characteristic vector field and $I_{n}=\left \{  1,2,...,n\right \}  $.Then
$\left \{  \theta,\theta^{\alpha},\theta^{\bar{\alpha}}\right \}  $, the coframe
dual to $\left \{  T,Z_{\alpha},Z_{\bar{\alpha}}\right \}  $, satisfies
\[
d\theta=ih_{\alpha \overline{\beta}}\theta^{\alpha}\wedge \theta^{\overline
{\beta}}%
\]
for some positive definite hermitian matrix of functions $(h_{\alpha \bar
{\beta}})$. If we have this contact structure, we call such $M$ a
pseudohermitian $(2n+1)$-manifold or strictly pseudoconvex CR $(2n+1)$%
-manifold as well.

The Levi form $\left \langle \ ,\  \right \rangle _{L_{\theta}}$ is the Hermitian
form on $T_{1,0}M$ defined by%
\[
\left \langle Z,W\right \rangle _{L_{\theta}}=-i\left \langle d\theta
,Z\wedge \overline{W}\right \rangle .
\]

We can extend $\left \langle \ ,\  \right \rangle _{L_{\theta}}$ to $T_{0,1}M$ by
defining $\left \langle \overline{Z},\overline{W}\right \rangle _{L_{\theta}%
}=\overline{\left \langle Z,W\right \rangle }_{L_{\theta}}$ for all $Z,W\in
T_{1,0}M$. The Levi form induces naturally a Hermitian form on the dual bundle
of $T_{1,0}M$, denoted by $\left \langle \ ,\  \right \rangle _{L_{\theta}^{\ast
}}$, and hence on all the induced tensor bundles. Integrating the Hermitian
form (when acting on sections) over $M$ with respect to the volume form
$d\mu=\theta \wedge(d\theta)^{n}$, we get an inner product on the space of
sections of each tensor bundle.

The pseudohermitian connection of $(J,\theta)$ is the connection $\nabla$ on
$TM\otimes \mathbb{C}$ (and extended to tensors) given in terms of a local
frame $Z_{\alpha}\in T_{1,0}M$ by%

\[
\nabla Z_{\alpha}=\omega_{\alpha}{}^{\beta}\otimes Z_{\beta},\quad \nabla
Z_{\bar{\alpha}}=\omega_{\bar{\alpha}}{}^{\bar{\beta}}\otimes Z_{\bar{\beta}%
},\quad \nabla T=0,
\]
where $\omega_{\alpha}{}^{\beta}$ are the $1$-forms uniquely determined by the
following equations :%

\[
\left \{
\begin{array}
[c]{l}%
d\theta^{\alpha}+\omega_{\beta}^{\alpha}\wedge \theta^{\beta}=\theta \wedge
\tau^{\alpha}\\
\tau_{\alpha}\wedge \theta^{\alpha}=0\\
\omega_{\beta}^{\alpha}+\omega_{\overline{\alpha}}^{\overline{\beta}}=0
\end{array}
\right.  .
\]

We can write (by Cartan lemma) $\tau_{\alpha}=A_{\alpha \gamma}\theta^{\gamma}$
with $A_{\alpha \gamma}=A_{\gamma \alpha}$. The curvature of Tanaka-Webster
connection, expressed in terms of the coframe $\{ \theta=\theta^{0}%
,\theta^{\alpha},\theta^{\bar{\alpha}}\}$, is%
\[
\left \{
\begin{array}
[c]{l}%
\Pi_{\beta}^{\alpha}=\overline{\Pi_{\overline{\beta}}^{\overline{\alpha}}%
}=d\omega_{\beta}^{\alpha}+\omega_{\gamma}^{\alpha}\wedge \omega_{\beta
}^{\gamma}\\
\Pi_{0}^{\alpha}=\Pi_{\alpha}^{0}=\Pi_{0}^{\overline{\alpha}}=\Pi
_{\overline{\alpha}}^{0}=\Pi_{0}^{0}=0
\end{array}
\right.  .
\]

Webster showed that $\Pi_{\beta}{}^{\alpha}$ can be written%
\[
\Omega_{\beta}^{\alpha}=\Pi_{\beta}^{\alpha}+i\tau^{\alpha}\wedge \theta
_{\beta}-i\theta^{\alpha}\wedge \tau_{\beta}=R_{\beta \text{ \ }\gamma
\overline{\delta}}^{\text{ \ }\alpha}\theta^{\gamma}\wedge \theta
^{\overline{\delta}}+W_{\beta \gamma}^{\alpha}\theta^{\gamma}\wedge
\theta-W_{\beta \overline{\gamma}}^{\alpha}\theta^{\overline{\gamma}}%
\wedge \theta
\]
where the coefficients satisfy
\[
R_{\beta \bar{\alpha}\rho \bar{\sigma}}=\overline{R_{\alpha \bar{\beta}\sigma
\bar{\rho}}}=R_{\bar{\alpha}\beta \bar{\sigma}\rho}=R_{\rho \bar{\alpha}%
\beta \bar{\sigma}},\  \  \ W_{\beta \bar{\alpha}\gamma}=W_{\gamma \bar{\alpha
}\beta}.
\]

Here $R_{\beta \text{ \ }\gamma \overline{\delta}}^{\text{ \ }\alpha}$ is the
pseudohermitian curvature tensor, $R_{\alpha \bar{\beta}}=R_{\gamma}{}^{\gamma
}{}_{\alpha \bar{\beta}}$ is the pseudohermitian Ricci curvature tensor and
$A_{\alpha \beta}$\ is the pseudohermitian torsion. $R=h^{\alpha \overline
{\beta}}R_{\alpha \bar{\beta}}$ denotes the pseudohermitian scalar curvature.
Besides, we define the pseudohermitian bisectional curvature%
\[
R_{\alpha \bar{\alpha}\beta \overline{\beta}}(X,Y)=R_{\alpha \bar{\alpha}%
\beta \overline{\beta}}X_{\alpha}X_{\overline{\alpha}}Y_{\beta}Y_{\bar{\beta}%
},
\]
the bitorsion tensor
\[
T_{\alpha \overline{\beta}}(X,Y)=\frac{1}{i}(A_{\alpha \gamma}X^{\gamma
}Y_{\overline{\beta}}-A_{\overline{\beta}\overline{\gamma}}X^{\overline
{\gamma}}Y_{\alpha}),
\]
the torsion tensor%
\[
Tor\left(  X,Y\right)  =tr\left(  T_{\alpha \overline{\beta}}\right)  =\frac
{1}{i}(A_{\alpha \beta}X^{\beta}Y^{\alpha}-A_{\overline{\alpha}\overline{\beta
}}X^{\overline{\beta}}Y^{\overline{\alpha}}),
\]
and the tensor
\[
\left(  \operatorname{div}A\right)  ^{2}\left(  X,Y\right)  =A_{\alpha
\gamma,\overline{\gamma}}A_{\overline{\beta}\overline{\delta},\delta}%
X^{\alpha}Y^{\overline{\beta}}%
\]
where $X=X^{\alpha}Z_{\alpha},Y=Y^{\alpha}Z_{\alpha}$ in $T_{1,0}M$.

We will denote the components of the covariant derivatives with indices
preceded by comma. The indices $\{0,\alpha,\bar{\alpha}\}$ indicate
derivatives with respect to $\{T,Z_{\alpha},Z_{\bar{\alpha}}\}$. For
derivatives of a real-valued function, we will often omit the comma, for
instance, $u_{\alpha}=Z_{\alpha}u,\ u_{\alpha \bar{\beta}}=Z_{\bar{\beta}%
}Z_{\alpha}u-\omega_{\alpha}{}^{\gamma}(Z_{\bar{\beta}})Z_{\gamma}u$. The
subgradient $\nabla_{b}\varphi$ of a smooth real-valued function $\varphi$ is
defined by
\[
\left \langle \nabla_{b}\varphi,Z\right \rangle _{L_{\theta}}=Z\varphi
\]
for $Z\in \Gamma \left(  \xi \right)  $ where $\Gamma \left(  \xi \right)  $
denotes the family of all smooth vector fields tangent to the conact plane
$\xi$. We could locally write the subgradient $\nabla_{b}\varphi$ as
\[
\nabla_{b}u=u_{\bar{\alpha}}Z_{\alpha}+u_{\alpha}Z_{\bar{\alpha}}.
\]

Accordingly, we could define the subhessian $Hess_{b}$ as the complex linear
map%
\[
Hess_{b}:T_{1,0}M\oplus T_{0,1}M\longrightarrow T_{1,0}M\oplus T_{0,1}M
\]
by%
\[
\left(  Hess_{b}\varphi \right)  Z=\nabla_{Z}\nabla_{b}\varphi
\]
for $Z\in \Gamma \left(  \xi \right)  $ and a smooth real-valued function
$\varphi$.

Next we recall the following commutation relations. (see \cite{l}) Let
$\varphi$ be a smooth real-valued function, $\sigma=\sigma_{\alpha}%
\theta^{\alpha}$ be a $\left(  1,0\right)  $-form and $\varphi_{0}=T\varphi,$
then we have%
\[
\left \{
\begin{array}
[c]{cll}%
\varphi_{\alpha \beta} & = & \varphi_{\beta \alpha},\\
\varphi_{\alpha \bar{\beta}}-\varphi_{\bar{\beta}\alpha} & = & ih_{\alpha
\overline{\beta}}\varphi_{0}\\
\varphi_{0\alpha}-\varphi_{\alpha0} & = & A_{\alpha \beta}\varphi^{\beta}\\
\sigma_{\alpha,\beta \gamma}-\sigma_{\alpha,\gamma \beta} & = & i\left(
A_{\alpha \gamma}\sigma_{\beta}-A_{\alpha \beta}\sigma_{\gamma}\right) \\
\sigma_{\alpha,\overline{\beta}\overline{\gamma}}-\sigma_{\alpha
,\overline{\gamma}\overline{\beta}} & = & -i(h_{\alpha \overline{\gamma}%
}A_{\overline{\beta}}^{\delta}\sigma_{\delta}-h_{\alpha \overline{\beta}%
}A_{\overline{\gamma}}^{\delta}\sigma_{\delta})\\
\sigma_{\alpha,\beta \overline{\gamma}}-\sigma_{\alpha,\overline{\gamma}\beta}
& = & ih_{\beta \overline{\gamma}}\sigma_{\alpha,}{}_{0}+R_{\alpha}^{\text{
}\delta}{}_{\beta \overline{\gamma}}\sigma_{\delta}\\
\sigma_{\alpha,0\beta}-\sigma_{\alpha,\beta0} & = & A^{\overline{\gamma}}%
{}_{\beta}\sigma_{\alpha,\bar{\gamma}}-A_{\alpha \beta,\bar{\gamma}}%
\sigma^{\overline{\gamma}}\\
\sigma_{\alpha,0\bar{\beta}}-\sigma_{\alpha,\bar{\beta}0} & = & A^{\gamma}%
{}_{\bar{\beta}}\sigma_{\alpha,\gamma}+A_{\bar{\gamma}\bar{\beta},\alpha
}\sigma^{\overline{\gamma}}%
\end{array}
\right.  .
\]

Last, we would introduce the concept of the adapted metric. A family of
Webster adapted metrics $h_{\epsilon}$ of a strictly pseudoconvex CR
$(2n+1)$-manifold $\left(  M,J,\theta \right)  $ are the Riemannian metrics%
\begin{equation}
h_{\epsilon}=h+\frac{1}{\epsilon^{2}}\theta^{2} \label{2017}%
\end{equation}
where $h=\left \langle \ ,\  \right \rangle _{L_{\theta}}$ is the Levi metric. In
\cite{cc}, they show that
\[
\Delta_{\epsilon}g=2\Delta_{b}g+\epsilon^{2}T^{2}g
\]
for any real-valued smooth function $g$ and%
\[
d\mu_{\epsilon}=\frac{1}{\epsilon2^{n}n!}d\mu
\]
where $d\mu \left(  y\right)  =dy=\theta \wedge \left(  d\theta \right)  ^{n}$ is
a volume form on $\left(  M,J,\theta \right)  $. Here $\Delta_{\epsilon}$ and
$d\mu_{\epsilon}$ are the Riemannian Laplace operator and the Riemannian
volume element with respective to the adapted Riemannian metric $h_{\epsilon}%
$, respectively. In this paper, we choose $\epsilon=\frac{1}{2^{n-1}n!}$such
that $d\mu_{\epsilon}=\frac{1}{2}d\mu$.

\section{The CR heat equation}

In this section, we will derive the essential fact about the existence of
solutions to the CR heat equation on a complete noncompact pseudohermitian
$(2n+1)$-manifold. First of all, we give the following lemmas so that one can
establish the existence of solutions to the CR heat equation. It is also very
useful to study the CR Poincar\'{e}-Lelong equation (\cite{cchl}).

We will use the semigroup method in \cite{bbgm}. It is known that the heat
semigroup $\left(  P_{t}\right)  _{t\geq0}$ is given by
\[
P_{t}=\int_{0}^{\infty}e^{-\lambda t}dE_{\lambda}%
\]
for the spectral decomposition of $\Delta_{b}=-\int_{0}^{\infty}\lambda
dE_{\lambda}$ in $L^{2}\left(  M\right)  $. It is a one-parameter family of
bounded operators on $L^{2}\left(  M\right)  .$ We denote%
\[
P_{t}f\left(  x\right)  =\int_{M}H\left(  x,y,t\right)  f\left(  y\right)
d\mu \left(  y\right)  .
\]
Here $H\left(  x,y,t\right)  >0$ is the so-called symmetric heat kernel
associated to $P_{t}$. Due to hypoellipticity of $\Delta_{b},$ the function
$\left(  x,t\right)  \rightarrow P_{t}f\left(  x\right)  $ is smooth on
$M\times \left(  0,\infty \right)  ,\ f\in C_{0}^{\infty}(M).$ Moreover%
\[
u\left(  x,t\right)  =P_{t}f(x)
\]
is a solution of CR heat equation%
\[
\left \{
\begin{array}
[c]{l}%
\frac{\partial u}{\partial t}=\Delta_{b}u\\
u\left(  x,0\right)  =f\left(  x\right)
\end{array}
\right.  .
\]
Under these settings, we're going to prove some results derived by estimating
the CR heat kernel $H\left(  x,y,t\right)  $. In the following, $V_{x}(r)$
denotes the volume of the geodesic ball $B_{x}(r)$ with respect to the
Carnot-Carath\'{e}odory distance $r(x,y)$ between $x$ and $y$, and
$r(x)=r(x,o)$ where $o\in M$ is a fixed point.

\begin{theorem}
\label{T31} Let $(M,J,\theta)$ be a complete noncompact pseudohermitian
$(2n+1)$-manifold with nonnegative pseudohermitian Ricci curvature tensors and
vanishing pseudohermitian torsion. If $f$ is a continuous function on $M$ such
that
\[
\frac{1}{V(B_{o}(r))}%
{\displaystyle \int \limits_{B_{o}(r)}}
|f|(x)dx\leq exp(ar^{2}+b)
\]
for some positive constants $a>0$ and $b>0$, then the following initial value
problem%
\begin{equation}
\left \{
\begin{array}
[c]{l}%
(\frac{\partial}{\partial t}-\triangle_{b})v(x,t)=0\\
v(x,0)=f(x)
\end{array}
\right.  \label{313}%
\end{equation}
has a solution on $M\times(0,\frac{C}{a}]$; moreover,
\[
v(x,t)=%
{\displaystyle \int \limits_{M}}
H(x,y,t)f(y)dy,
\]
where $H(x,y,t)$ is the Heat kernel on $(M,J,\theta)$.
\end{theorem}

The proof is based on the effect estimate on the CR heat kernel, CR volume
doubling property and $L^{p}$ submean value inequality which do not need the
stronger assumption on vanishing pseudohermitian torsion as in Proposition
\ref{p41}. However, for the later purpose as in the following sections, we do
need the condition of vanishing torsion, in which the heat kernel estimate is
obtained also in \cite{bbgm}. Then here we focus on the case of a complete
noncompact pseudohermitian $(2n+1)$-manifold of vanishing torsion only.

\begin{proposition}
(\cite{ccht}, \cite{bbgm}) \label{p41} Let $(M,J,\theta)$ be a complete
pseudohermitian $\left(  2n+1\right)  $-manifold with the pseudohermitian
Ricci curvature
\[
Ric\left(  X,X\right)  \geq k_{0}\left \langle X,X\right \rangle _{L_{\theta}}%
\]
and \
\[
\sup_{i,j\in I_{n}}\left \vert A_{ij}\right \vert \leq k_{1}<\infty \text{
\textrm{and }}\sup_{i,j\in I_{n}}|A_{ij,\bar{\imath}}|^{2}\leq k_{2}<\infty,
\]
for $X=X^{\alpha}Z_{\alpha}\in T_{1,0}M$ and $k_{0},\ k_{1},\ k_{2}$ are
constants with $k_{1},\ k_{2}\geq0.$ Then

(i) There exist positive constants $C_{3},C_{4},C_{5}$ such that for $x,\ y\in
M,\ t>0$%
\begin{equation}
H\left(  x,y,t\right)  \leq \frac{C_{3}}{V_{x}\left(  \sqrt{t}\right)
^{\frac{1}{2}}V_{y}\left(  \sqrt{t}\right)  ^{\frac{1}{2}}}\exp(-C_{5}%
\frac{d_{cc}^{2}\left(  x,y\right)  }{t}+C_{4}\kappa t). \label{2016a}%
\end{equation}

(ii) There exist positive constants $C_{6},C_{7},C_{8}$ such that for
$x,\ y\in M,\ t>0$
\begin{equation}
H(x,y,t)\geq \frac{C_{6}}{V_{x}\left(  \sqrt{t}\right)  }\exp[-C_{7}%
\frac{d_{cc}{}^{2}\left(  x,y\right)  }{t}-C_{8}\kappa(t+d_{cc}^{2}(x,y))].
\label{2016b}%
\end{equation}

(iii) There exist positive constants $C_{9},C_{10}$ such that for $0<s<t$
\begin{equation}
\frac{H(x,x,s)}{H(x,x,t)}\, \leq \, \left(  \frac{t}{s}\right)  ^{C_{9}%
}e^{C_{10}\kappa(t-s)}. \label{2016c}%
\end{equation}

(iv) (CR Volume Doubling Property) For any $\sigma>1$, then there exist a
positive constant $C_{1}$ such that
\[
V_{x}\left(  \sigma \rho \right)  \leq C_{1}\sigma^{2C_{9}}e^{(C_{1}\sigma
^{2}+C_{8})\kappa \rho^{2}}V_{x}\left(  \rho \right)  .
\]
Here $\kappa=\kappa(k_{0},k_{1},k_{2})\geq0$. In particular, if $(M,J,\theta)$
admits nonnegative pseudohermitian Ricci curvature tensor and vanishing
pseudohermitian torsion, then those constants $C_{j}$ chosen as precedes all
depend on the CR dimension $n$ and the constant $\kappa=0$.
\end{proposition}

\begin{lemma}
\label{l31} Let $(M,J,\theta)$ be a complete noncompact pseudohermitian
$(2n+1)$-manifold of vanishing torsion with nonnegative pseudohermitian Ricci
curvature tensors. Assume that $u$ is defined by
\[
u(x,t)=\int_{M}H(x,y,t)f(y)dy,
\]
on $M\times \lbrack0,T]$ for some $T>0.$ Here $f$ is a nonnegative function
with
\begin{equation}
\lim_{r\rightarrow+\infty}exp(-\frac{C_{5}\left(  n\right)  r^{2}}{2t}%
)\int_{B_{o}(r)}f(x)dx=0. \label{301}%
\end{equation}
Then for any $0<t\leq r^{2}$, and $p\geq1$,%
\begin{equation}%
\begin{array}
[c]{ccl}%
\frac{1}{V(B_{o}(r))}\int_{B_{o}(r)}u^{p}(x,t)dx & \leq & C(n,p)[\frac
{1}{V(B_{o}(4r))}\int_{B_{o}(4r)}f^{p}(x)dx\\
&  & +(\frac{1}{t}\int_{4r}^{\infty}exp(-\frac{C_{5}(n)s^{2}}{4t})s(\frac
{1}{V(B_{o}(s))}\int_{B_{o}(s)}f(x)dx)ds)^{p}].
\end{array}
\label{302}%
\end{equation}

\end{lemma}

\begin{proof}
For any $p\geq1$ and $r\geq \sqrt{t}$,%
\begin{equation}%
\begin{array}
[c]{ccl}%
\int_{B_{o}(r)}u^{p}(x,t)dx & = & \int_{B_{o}(r)}(\int_{M}H(x,y,t)f(y)dy)^{p}%
dx\\
& \leq & 2^{p-1}[\int_{B_{o}(r)}(\int_{B_{o}(4r)}H(x,y,t)f(y)dy)^{p}dx\\
&  & +\int_{B_{o}(r)}(\int_{M\backslash B_{o}(4r)}H(x,y,t)f(y)dy)^{p}dx].
\end{array}
\label{303}%
\end{equation}
We first estimate the second term on the right hand side in $\left(
\ref{303}\right)  $. Now for $x\in B_{o}(r)$ and $y\in M\backslash B_{o}(4r)$,
we have $r(x,y)\geq \frac{3}{4}r(y)$. By the estimates of the CR heat kernel as
in Proposition \ref{p41}, we have%
\begin{equation}%
\begin{array}
[c]{cl}
& \int_{M\backslash B_{o}(4r)}H(x,y,t)f(y)dy\\
\leq & C\int_{M\backslash B_{o}(4r)}\frac{1}{V(B_{x}(\sqrt{t}))}%
exp(-\frac{C_{5}r^{2}(x,y)}{t})f(y)dy\\
\leq & C\int_{M\backslash B_{o}(4r)}\frac{1}{V(B_{x}(r+\sqrt{t}))}%
(\frac{r+\sqrt{t}}{\sqrt{t}})^{2C_{9}}exp(-\frac{C_{5}r^{2}(x,y)}{t})f(y)dy\\
\leq & \frac{C}{V(B_{o}(\sqrt{t}))}\int_{M\backslash B_{o}(4r)}(\frac
{r+\sqrt{t}}{\sqrt{t}})^{2C_{9}}exp(-\frac{C_{5}r^{2}(x,y)}{t})f(y)dy\\
\leq & \frac{C}{V(B_{o}(\sqrt{t}))}(\frac{r+\sqrt{t}}{\sqrt{t}})^{2C_{9}}%
\int_{4r}^{\infty}[exp(-\frac{C_{5}s^{2}}{2t})(\int_{\partial B_{o}(s)}f)]ds\\
\leq & \frac{C}{V(B_{o}(\sqrt{t}))}(\frac{r+\sqrt{t}}{\sqrt{t}})^{2C_{9}}%
\int_{4r}^{\infty}[exp(-\frac{C_{5}s^{2}}{2t})(\int_{B_{o}(s)}f)]d(\frac
{s^{2}}{t})\\
\leq & C(\frac{r+\sqrt{t}}{\sqrt{t}})^{2C_{9}}\int_{4r}^{\infty}[\frac
{V(B_{o}(s))}{V(B_{o}(\sqrt{t}))}exp(-\frac{C_{5}s^{2}}{2t})(\frac{1}%
{V(B_{o}(s))}\int_{B_{o}(s)}f)]d(\frac{s^{2}}{t})\\
\leq & C\int_{4r}^{\infty}(\frac{s}{\sqrt{t}})^{4C_{9}}exp(-\frac{C_{5}s^{2}%
}{2t})(\frac{1}{V(B_{o}(s))}\int_{B_{o}(s)}f)d(\frac{s^{2}}{t})\\
\leq & C\frac{1}{t}\int_{4r}^{\infty}exp(-\frac{C_{5}s^{2}}{4t})s(\frac
{1}{V(B_{o}(s))}\int_{B_{o}(s)}f)ds.
\end{array}
\label{304}%
\end{equation}
where $C_{9}$ is the constant as in Proposition \ref{p41}. Here, besides
utilizing the inequality $r(x,y)\geq \frac{3}{4}r(y)$, we have used the volume
doubling property and the assumption (\ref{301}) when performing integration
by parts in the fifth inequality. Note that when separating $M\backslash
B_{o}(4r)$ into the shells, the concept of the Legendrian normal has come into
our case (see \cite{ccw}). From now on, once we deal with such process of the
integration, we keep the idea in mind.

As for the first term on the right hand side in $\left(  \ref{303}\right)  $,
by H\"{o}lder's inequality and the fact (\cite{bbgm}) that
\[
\int_{M}H(x,y,t)dy=1,
\]
we have
\[
(\int_{B_{o}(4r)}H(x,y,t)f(y)dy)^{p}\leq \int_{B_{o}(4r)}H(x,y,t)f^{p}(y)dy.
\]
Hence%
\begin{equation}%
\begin{array}
[c]{ccl}%
\int_{B_{o}(r)}(\int_{B_{o}(4r)}H(x,y,t))f(y)dy)^{p}dx & \leq & \int
_{B_{o}(r)}\int_{B_{o}(4r)}H(x,y,t)f^{p}(y)dydx\\
& \leq & \int_{B_{o}(4r)}f^{p}(y)(\int_{B_{o}(r)}H(x,y,t)dx)dy\\
& \leq & \int_{B_{o}(4r)}f^{p}(y)dy.
\end{array}
\label{305}%
\end{equation}
From (\ref{303})-(\ref{305}), we know that the Lemma holds.
\end{proof}

From the preceding lemma, we could control the upper bound of subsolutions to
the power of $p$ of the CR heat equation by its $L^{p}$-norm under the same
hypotheses. On the other hand, during the course of proving the existence of
solutions to the CR heat equation, we also need the $L^{p}$ submean value
inequality for $p\in \left(  0,\infty \right)  $, which will be shown in the Appendix.

\begin{proposition}
\label{P31} (CR $L^{p}$ submean value inequality ) Let $(M,J,\theta)$ be a
complete pseudohermitian $(2n+1)$-manifold of vanishing torsion with
nonnegative pseudohermitian Ricci curvature tensors. Let $Q(\tau
,x,r,s)=(s-\tau r^{2},s)\times B_{x}(r)$ and $Q_{\delta}(\tau,x,r,s)=(s-\delta
\tau r^{2},s)\times B_{x}(\delta r)$ for $\tau>0,$ $x\in M,$ $r>0,$ $s\in R,$
$\delta \in(0,1)$. If $\tau>0$ and $p\in \left(  0,+\infty \right)  $ are given,
then there exists a constant $A(\tau,\nu,p)>0$ such that if $u$ is a positive
subsolution to the CR heat equation $(\frac{\partial}{\partial t}%
-\triangle_{b})u\leq0$ in $Q(\tau,x,r,s)$, then, for $0<\delta<\delta^{\prime
}\leq1$,%
\[
\sup_{Q_{\delta}}\{u^{p}\} \leq \frac{A(\tau,\nu,p)}{\left(  \delta^{\prime
}-\delta \right)  ^{2+\nu}r^{2}V\left(  B_{x}(r)\right)  }\int_{Q_{\delta
^{\prime}}}u^{p}d\overline{\mu},
\]
where $d\overline{\mu}=d\mu dt$. \ Here the constant $\nu=3C_{9}$ is the
exponential constant in CR Sobolev inequality (\cite{ccht})%
\begin{equation}
(\int_{B_{x}(\rho)}|\varphi|^{\frac{2\nu}{\nu-2}}d\mu)^{^{\frac{\nu-2}{\nu}}%
}\leq C_{s}\rho^{2}V\left(  B_{x}(r)\right)  ^{-\frac{2}{\nu}}[\int
_{B_{x}(\rho)}|\nabla_{b}\varphi|^{2}d\mu+\rho^{-2}\int_{B_{x}(\rho)}%
\varphi^{2}d\mu] \label{p31}%
\end{equation}
for any $\varphi \in C_{c}^{\infty}(B_{x}(\rho))$, $x\in M.$
\end{proposition}

Now we are ready to prove \textbf{Theorem \ref{T31} : }

\begin{proof}
For all $j\geq1$, let $0\leq \varphi_{j}\leq1$ be a smooth cut-off function
such that $\varphi_{j}=1$ on $B_{o}(j)$ and $\varphi_{j}=0$ on $M\backslash
B_{o}(2j)$. Let $f_{j}=\varphi_{j}f$ be continuous with compact support. Hence
one can solve $\left(  \ref{313}\right)  $ with the initial value $f_{j}$ for
all time. The solution $v_{j}$ is given by
\[
v_{j}(x,t)=\int_{M}H(x,y,t)f_{j}(y)dy
\]
for $(x,t)\in M\times(0,\infty)$. The existence and uniqueness of such $v_{j}$
of the form could be found in \cite{li}. By Lemma \ref{l31}, for $0<t\leq
\min \left \{  r^{2},\frac{C_{5}}{8a}\right \}  $, we have%
\[%
\begin{array}
[c]{ccl}%
\frac{1}{V(B_{o}(r))}\int_{B_{o}(r)}|v_{j}|dx & \leq & \frac{1}{V(B_{o}%
(r))}\int_{B_{o}(r)}\int_{M}H(x,y,t)|f_{j}(y)|dydx\\
& \leq & C[\frac{1}{V(B_{o}(4r))}\int_{B_{o}(4r)}|f_{j}(y)|dy\\
&  & +\frac{1}{t}\int_{4r}^{\infty}exp(-\frac{C_{5}s^{2}}{4t})s(\frac
{1}{V(B_{o}(s))}\int_{B_{o}(s)}f_{j})ds]\\
& \leq & Ce^{b}[exp(16ar^{2})+\int_{4r}^{\infty}exp(-\frac{C_{5}s^{2}}%
{8t})d(\frac{s^{2}}{t})]\\
& \leq & Ce^{b}\left[  exp(16ar^{2})+1\right]  .
\end{array}
\]

Assume $\frac{C_{5}}{8a}\leq R^{2}$. Let $(\tau,x,r,s,p,\delta,\delta^{\prime
})=(2,o,R,\frac{C_{5}}{8a},1,\frac{1}{2},1)$ in the CR $L^{p}$ submean value
inequality, and we know that $|v_{j}|$ is a subsolution to the CR heat
equation, then we have%
\[%
\begin{array}
[c]{ccl}%
\underset{B_{o}(\frac{R}{2})\times(0,\frac{C_{5}}{8a})}{\sup}|v_{j}| & \leq &
\frac{A\left(  2,\nu,1\right)  }{\left(  \frac{1}{2}\right)  ^{2+\nu}%
R^{2}V\left(  B_{o}\left(  R\right)  \right)  }%
{\displaystyle \int \limits_{Q_{1}}}
\left \vert v_{j}\right \vert d\overline{\mu}\\
& \leq & 2^{3+\nu}A\left(  2,\nu,1\right)  \frac{1}{V(B_{o}(R))}%
{\displaystyle \int \limits_{B_{o}(R)}}
|v_{j}|dx\\
& \leq & Cexp(16aR^{2}+b).
\end{array}
\]

From this, it's easy to see that, after passing to a subsequence, $\left \{
v_{j}\right \}  _{j\in%
\mathbb{N}
}$ together with their derivatives uniformly converge on compact sets on
$M\times(0,\frac{C_{5}}{8a})$ to a solution $v$ of the CR heat equation.
Moreover, for any $(x,t)\in B_{o}\left(  \frac{j}{4}\right)  \times
(0,\frac{C_{5}}{8a})$ as in $\left(  \ref{304}\right)  $, we have%
\[%
\begin{array}
[c]{ccl}%
\left \vert
{\displaystyle \int \limits_{M}}
H(x,y,t)f\left(  y\right)  dy-v_{j}(x,t)\right \vert  & \leq & \left \vert
{\displaystyle \int \limits_{M}}
H(x,y,t)(f(y)-f_{j}(y))dy\right \vert \\
& \leq &
{\displaystyle \int \limits_{M\backslash B_{o}\left(  j\right)  }}
H(x,y,t)|f\left(  y\right)  |dy\\
& \leq & C\int_{j}^{\infty}exp(-\frac{C_{5}s^{2}}{4t})d(\frac{s^{2}}{t})\\
& \leq & C\int_{\frac{j^{2}}{t}}^{\infty}exp(-\frac{C_{5}\tau}{4})d\tau.
\end{array}
\]

Let $j$ tend to $+\infty$, we have
\[
v(x,t)=%
{\displaystyle \int \limits_{M}}
H(x,y,t)f(y)dy
\]

on $M\times \left(  0,\frac{C}{a}\right]  $. This completes the proof.
\end{proof}

\section{Rough Dimension Estimates}

After settling the existence of solutions to the CR heat equation, we will
give the asymptotic behavior of the solutions. It's a crucial step on the
course of proving the sharp dimension estimate. And it's worth to note that
because of the following lemma, we could drop the assumption that $M$ is of
maximum volume growth in the hypotheses of the sharp dimension estimate.
That's why the statement about the sharp dimension estimate in \cite{cfyz}
could hold without the assumption of the maximum volume growth as in \cite{n}.

\begin{lemma}
\label{l4.2} Let $(M,J,\theta)$ be a complete pseudohermitian $(2n+1)$%
-manifold of vanishing torsion with nonnegative pseudohermitian Ricci
curvature tensors. Assume that $v\left(  x,t\right)  $ is a solution to the CR
heat equation%
\[
\left(  \frac{\partial}{\partial t}-\Delta_{b}\right)  v\left(  x,t\right)
=0
\]
on $M\times \left(  0,+\infty \right)  $ with the initial condition
\[
v\left(  x,0\right)  =2\log \left \vert f\left(  x\right)  \right \vert ,
\]
where $f\in \mathcal{O}_{d}^{CR}\left(  M\right)  .$ Then
\begin{equation}
\underset{t\rightarrow+\infty}{\lim \sup}\frac{v\left(  x,t\right)  }{\log
t}\leq d. \label{314}%
\end{equation}

\end{lemma}

\begin{proof}
From the fact $f\in \mathcal{O}_{d}^{CR}\left(  M\right)  $, we see that for
all $\epsilon>0$, there's a constant $C^{\prime}=C^{\prime}\left(
f,x,d,\epsilon \right)  >0$ such that%
\[
\left \vert f\left(  y\right)  \right \vert \leq C^{\prime}(1+r\left(
x,y\right)  )^{d+\epsilon}.
\]

Hereafter, in this proof, we would denote the positive constants by
$C^{\prime}=C^{\prime}\left(  f,x,d,\epsilon \right)  $ and $C=C\left(
n,d,\epsilon \right)  $. These constants may be different line by line.

Because
\[
\frac{1}{V(B_{o}(r))}%
{\displaystyle \int \limits_{B_{o}(r)}}
\log \left \vert f\left(  x\right)  \right \vert dx\leq exp(ar^{2}+b)
\]

for any $a>0$, then, from Theorem \ref{T31}, we know the solution $v\left(
x,t\right)  $ has the form%
\begin{equation}
v(x,t)=2%
{\displaystyle \int \limits_{M}}
H(x,y,t)\log \left \vert f\left(  y\right)  \right \vert dy \label{350}%
\end{equation}

on $M\times \left(  0,+\infty \right)  $.

Hence%
\[%
\begin{array}
[c]{ccl}%
v\left(  x,t\right)  & = & 2%
{\displaystyle \int \limits_{M}}
H(x,y,t)\log \left \vert f\left(  y\right)  \right \vert dy\\
& = &
{\displaystyle \int \limits_{\left[  r\left(  x,y\right)  \leq \sqrt{t}\right]
}}
H(x,y,t)\log(\left \vert f\left(  y\right)  \right \vert ^{2})dy+%
{\displaystyle \int \limits_{\left[  r\left(  x,y\right)  >\sqrt{t}\right]  }}
H(x,y,t)\log(\left \vert f\left(  y\right)  \right \vert ^{2})dy\\
& \leq &
{\displaystyle \int \limits_{\left[  r\left(  x,y\right)  \leq \sqrt{t}\right]
}}
H(x,y,t)\left(  \left(  d+\epsilon \right)  \log t+C^{\prime}\right)  dy\\
&  & +%
{\displaystyle \int \limits_{\left[  r\left(  x,y\right)  >\sqrt{t}\right]  }}
H(x,y,t)\left(  2\left(  d+\epsilon \right)  \log r\left(  x,y\right)
+C^{\prime}\right)  dy
\end{array}
\]
for $t>1$. By the CR heat kernel estimate and the CR volume doubling property
in \cite{ccht} and \cite{bbgm}, we obtain%
\[%
\begin{array}
[c]{cl}
& v\left(  x,t\right) \\
\leq & \left(  d+\epsilon \right)  \log t+C^{\prime}+%
{\displaystyle \int \limits_{\left[  r\left(  x,y\right)  >\sqrt{t}\right]  }}
\left(  d+\epsilon \right)  H\left(  x,y,t\right)  \log(\frac{r\left(
x,y\right)  ^{2}}{t})dy\\
\leq & C%
{\displaystyle \int \limits_{\left[  r\left(  x,y\right)  >\sqrt{t}\right]  }}
\frac{d+\epsilon}{V\left(  B_{x}\left(  \sqrt{t}\right)  \right)  }%
\exp(-C\frac{r\left(  x,y\right)  ^{2}}{2t})\log(\frac{r\left(  x,y\right)
^{2}}{t})dy+\left(  d+\epsilon \right)  \log t+C^{\prime}\\
\leq & C%
{\displaystyle \sum \limits_{k\geq0}}
{\displaystyle \int \limits_{\left[  2^{k}\sqrt{t}<r\left(  x,y\right)
\leq2^{k+1}\sqrt{t}\right]  }}
\frac{d+\epsilon}{V\left(  B_{x}\left(  \sqrt{t}\right)  \right)  }%
\exp(-C\frac{r\left(  x,y\right)  ^{2}}{2t})\log(\frac{r\left(  x,y\right)
^{2}}{t})dy+\left(  d+\epsilon \right)  \log t+C^{\prime}\\
\leq & C\left(  d+\epsilon \right)
{\displaystyle \sum \limits_{k\geq0}}
\left[  2^{2\left(  k+1\right)  C_{9}}\exp \left(  -2^{2k-1}C\right)  2\left(
k+1\right)  \right]  +\left(  d+\epsilon \right)  \log t+C^{\prime}\\
\leq & \left(  d+\epsilon \right)  \log t+C^{\prime}+C.
\end{array}
\]

Here the CR volume doubling property is used in the fourth inequality. Then we
are done.
\end{proof}

Now we give another lemma which is also essential for the proof of the rough
dimension estimate. In the proof of the following lemma, we would find that
it's closely related to the CR\ moment-type estimate as in \cite{cf}.

\begin{lemma}
\label{l43} Let $(M,J,\theta)$ be a complete pseudohermitian $(2n+1)$-manifold
of vanishing torsion with nonnegative pseudohermitian bisectional curvature
tensors. Assume that $v\left(  x,t\right)  $ is a solution to the CR heat
equation%
\begin{equation}
\left(  \frac{\partial}{\partial t}-\Delta_{b}\right)  v\left(  x,t\right)  =0
\label{319}%
\end{equation}
on $M\times \left(  0,+\infty \right)  $ with the initial condition%
\begin{equation}
v\left(  x,0\right)  =2\log \left \vert f\left(  x\right)  \right \vert
\label{320}%
\end{equation}
for $f\in \mathcal{O}_{d}^{CR}\left(  M\right)  $, and
\[
w\left(  x,t\right)  =\frac{\partial}{\partial t}v\left(  x,t\right)  .
\]
Then there exists a positive constant $C$ such that%
\begin{equation}
C(n)ord_{x}f\leq \lim_{t\rightarrow0^{+}}tw\left(  x,t\right)  . \label{316}%
\end{equation}

\end{lemma}

\begin{proof}
On account of the equation $\left(  \ref{319}\right)  $ and the initial
condition $\left(  \ref{320}\right)  $, we have%
\[
w\left(  x,t\right)  =2%
{\displaystyle \int \limits_{M}}
H(x,y,t)\Delta_{b}\log \left \vert f\left(  y\right)  \right \vert dy.
\]
by $\left(  \ref{350}\right)  $. So by the CR heat kernel estimate in
\cite{ccht}, we know%
\begin{equation}%
\begin{array}
[c]{ccl}%
w\left(  x,t\right)  & = & 2%
{\displaystyle \int \limits_{M}}
H(x,y,t)\Delta_{b}\log \left \vert f\left(  y\right)  \right \vert dy\\
& \geq & \frac{C}{V\left(  B_{x}\left(  \sqrt{t}\right)  \right)  }%
{\displaystyle \int \limits_{M}}
\exp(-C\frac{d\left(  x,y\right)  ^{2}}{t})\Delta_{b}\log \left \vert f\left(
y\right)  \right \vert dy\\
& \geq & \frac{C}{V\left(  B_{x}\left(  \sqrt{t}\right)  \right)  }%
{\displaystyle \int \limits_{\left[  0,\sqrt{t}\right]  }}
\exp \left(  -C\frac{r^{2}}{t}\right)  (%
{\displaystyle \int \limits_{\partial B_{x}\left(  r\right)  }}
\Delta_{b}\log \left \vert f\left(  s\right)  \right \vert d\sigma \left(
s\right)  )dr\\
& \geq & \frac{C}{V\left(  B_{x}\left(  \sqrt{t}\right)  \right)  }%
{\displaystyle \int \limits_{B_{x}\left(  \sqrt{t}\right)  }}
\Delta_{b}\log \left \vert f\left(  y\right)  \right \vert dy.
\end{array}
\label{322}%
\end{equation}

Furthermore, from the local property of the Sasakian manifolds and $Tf=0$, we
have, as $t<<1$,%
\begin{equation}%
\begin{array}
[c]{cl}
& \frac{C}{V\left(  B_{x}\left(  \sqrt{t}\right)  \right)  }%
{\displaystyle \int \limits_{B_{x}\left(  \sqrt{t}\right)  }}
\Delta_{b}\log \left \vert f\left(  y\right)  \right \vert dy\\
= & \frac{C}{V\left(  B_{x}\left(  \sqrt{t}\right)  \right)  }\int_{-\sqrt{t}%
}^{\sqrt{t}}\text{ }\int_{B_{\widetilde{x}}\left(  \sqrt{t-l^{2}}\right)
}\Delta_{\widetilde{y}}\log \left \vert f\left(  \left(  \widetilde{y},t\right)
\right)  \right \vert d\widetilde{y}dl\\
\geq & \frac{C}{t^{\frac{2n+1}{2}}}ord_{\widetilde{x}}\left(  \widetilde
{f}\right)  \int_{-\sqrt{t}}^{\sqrt{t}}\frac{V\left(  B_{\widetilde{x}}\left(
\sqrt{t-l^{2}}\right)  \right)  }{\left(  t-l^{2}\right)  }dl\text{ \ by
}\left(  \ref{323}\right) \\
= & \frac{C}{t^{\frac{2n+1}{2}}}ord_{x}\left(  f\right)  \int_{-\sqrt{t}%
}^{\sqrt{t}}\frac{V\left(  B_{\widetilde{x}}\left(  \sqrt{t-l^{2}}\right)
\right)  }{\left(  t-l^{2}\right)  }dl\\
\geq & \frac{C}{t^{\frac{2n+1}{2}}}ord_{x}\left(  f\right)  \int_{-\sqrt{t}%
}^{\sqrt{t}}\left(  t-l^{2}\right)  ^{n-1}dl\\
= & \frac{C}{t^{\frac{2n+1}{2}}}t^{n-\frac{1}{2}}ord_{x}\left(  f\right) \\
= & \frac{C}{t}ord_{x}\left(  f\right)  .
\end{array}
\label{342}%
\end{equation}

Here we have used the identity
\begin{equation}
ord_{\widetilde{x}}\left(  \widetilde{f}\right)  =\frac{1}{2n}\lim
_{r\rightarrow0^{+}}\frac{r^{2}}{V\left(  B_{\widetilde{x}}\left(  r\right)
\right)  }%
{\displaystyle \int \limits_{B_{\widetilde{x}}\left(  r\right)  }}
\widetilde{\Delta}\log \left \vert \widetilde{f}(\widetilde{y})\right \vert
d\widetilde{y} \label{323}%
\end{equation}
where $\widetilde{\Delta}$ denotes the Laplace operator on the slice
$\widetilde{D^{n}}$ with the trasversal K\"{a}hler structure and
$\widetilde{f}\left(  \widetilde{y}\right)  =f\left(  \widetilde{y},t\right)
$ for $t\ll1$. Also we adopt the adapted metric and
\[
ord_{x}\left(  f\right)  =ord_{\widetilde{x}}\left(  \widetilde{f}\right)  ,
\]
which comes from the fact $Z_{\alpha}f=\frac{\partial f}{\partial z^{\alpha}%
},\ f_{0}=0$ and utilize the equality%
\begin{equation}
\int_{-s}^{s}\left(  s^{2}-l^{2}\right)  ^{n-1}dl=s^{2n-1}\sqrt{\pi}%
\frac{\Gamma \left(  n\right)  }{\Gamma(n+\frac{1}{2})} \label{343}%
\end{equation}
in the third equality of $\left(  \ref{342}\right)  $ with $s=\sqrt{t}$ which
is derived as follows :%
\[%
\begin{array}
[c]{ccl}%
\int_{-s}^{s}\left(  s^{2}-l^{2}\right)  ^{n-1}dl & = & s^{2n-1}\int_{-1}%
^{1}\left(  1-x^{2}\right)  ^{n-1}dx\\
& = & 2s^{2n-1}\int_{0}^{1}\left(  1-x^{2}\right)  ^{n-1}dx\\
& = & s^{2n-1}\int_{0}^{1}\left(  1-z\right)  ^{n-1}\frac{1}{\sqrt{z}}dz\\
& = & s^{2n-1}B(\frac{1}{2},n)\\
& = & s^{2n-1}\frac{\Gamma \left(  \frac{1}{2}\right)  \Gamma \left(  n\right)
}{\Gamma \left(  n+\frac{1}{2}\right)  }\\
& = & s^{2n-1}\sqrt{\pi}\frac{\Gamma \left(  n\right)  }{\Gamma \left(
n+\frac{1}{2}\right)  }%
\end{array}
\]
where $B\left(  z,w\right)  $ and $\Gamma \left(  z\right)  $ denote the beta
function and the gamma function respectively.

Lemma \ref{l43} is accomplished.
\end{proof}

Note that, in the last proof we implicitly utilize the property of the
monotonicity of $tw\left(  x,t\right)  $ which would be explained in the one
of the next theorem. Prior to showing the sharp dimension estimate, we would
obtain the dimension estimates of the rough version as follows :

\begin{theorem}
Let $(M,J,\theta)$ be a complete noncompact pseudohermitian $(2n+1)$-manifold
of vanishing torsion with nonnegative pseudohermitian bisectional curvature.
Then there exists a positive constant $C\left(  n\right)  $ such that%
\begin{equation}
\dim_{%
\mathbb{C}
}\left(  \mathcal{O}_{d}^{CR}\left(  M\right)  \right)  \leq C\left(
n\right)  d^{n}. \label{315}%
\end{equation}

\end{theorem}

\begin{proof}
We adopt the notations in the proof of Lemma \ref{l43}. Set $w\left(
x,t\right)  =\frac{\partial}{\partial t}v\left(  x,t\right)  $. By the similar
deductions in \cite[Lemma 4.2.]{cf}, replaced the moment-type estimate in
\cite[(3.10) and (5.12)]{cf} \ by $\left(  \ref{314}\right)  $, we came out
with \cite[(5.14)]{cf} and then
\begin{equation}
\frac{\partial}{\partial t}\left(  tw\left(  x,t\right)  \right)  \geq0.
\label{324}%
\end{equation}

Note that the nonnegativity of the pseudohermitian bisectional curvature is
necessary here.

We first show that%
\begin{equation}
ord_{x}f\leq Cd. \label{318}%
\end{equation}

Basically, if we could prove%
\begin{equation}
tw\left(  x,t\right)  \leq Cd \label{317}%
\end{equation}
for any $t>1$, then we have $\left(  \ref{318}\right)  $ by Lemma \ref{l43}.
We observe that on account of the definition of $\mathcal{O}_{d}^{CR}\left(
M\right)  $ for all $x\in M$, there exists a constant $\widetilde
{C}=\widetilde{C}\left(  f,x,d,\epsilon \right)  >0$ such that%
\[
\log^{+}\left \vert f\left(  y\right)  \right \vert \leq \widetilde{C}+d\left(
\log \left(  1+d\left(  x,y\right)  \right)  \right)  .
\]

Accordingly we have, by the CR hea kernel estimate and the CR volume doubling
property,%
\begin{equation}%
\begin{array}
[c]{cl}
& v\left(  x,t\right) \\
\leq & \frac{C}{V\left(  B_{x}\left(  \sqrt{t}\right)  \right)  }\int
_{0}^{+\infty}[\exp \left(  -C\frac{r^{2}}{2t}\right)  (%
{\displaystyle \int \limits_{\partial B_{x}\left(  r\right)  }}
\log^{+}\left \vert f\left(  s\right)  \right \vert d\sigma \left(  s\right)
)]dr\\
\leq & \frac{C}{V\left(  B_{x}\left(  \sqrt{t}\right)  \right)  }\int
_{\sqrt{t}}^{+\infty}[\exp \left(  -C\frac{r^{2}}{2t}\right)  (%
{\displaystyle \int \limits_{\partial B_{x}\left(  r\right)  }}
\log^{+}\left \vert f\left(  s\right)  \right \vert d\sigma \left(  s\right)
)]dr+Cd\log \left(  1+t\right) \\
\leq & \frac{C}{V\left(  B_{x}\left(  \sqrt{t}\right)  \right)  }\int
_{\sqrt{t}}^{+\infty}[\exp \left(  -C\frac{r^{2}}{2t}\right)  \frac{Cr}{t}(%
{\displaystyle \int \limits_{B_{x}\left(  r\right)  }}
\log^{+}\left \vert f\left(  y\right)  \right \vert dy)]dr+Cd\log \left(
1+t\right) \\
\leq & C\int_{\sqrt{t}}^{+\infty}[(\frac{r}{\sqrt{t}})^{2C_{9}}\exp \left(
-C\frac{r^{2}}{2t}\right)  \frac{Cr}{t}d\log \left(  1+r\right)  ]dr+Cd\log
\left(  1+t\right) \\
\leq & C\int_{t}^{+\infty}[\frac{\tau^{C_{9}}}{t^{C_{9}+1}}\exp \left(
-C\frac{\tau}{2t}\right)  d\log \left(  1+\tau \right)  ]d\tau+Cd\log \left(
1+t\right) \\
\leq & C\int_{0}^{+\infty}\left[  s^{C_{9}}\exp \left(  -s\right)
sd\log \left(  1+t\right)  \right]  ds+Cd\log \left(  1+t\right) \\
\leq & Cd\log \left(  1+t\right)  +Cd\log \left(  1+t\right)  \Gamma \left(
C_{9}+2\right) \\
\leq & Cd\log t
\end{array}
\label{321}%
\end{equation}
for $t>1$. From now on, we fix the constant $C$ in $\left(  \ref{321}\right)
$. With the help of the inequality $\left(  \ref{321}\right)  $, we claim that
the inequality $\left(  \ref{317}\right)  $ holds; for if there's some small
positive constant $\epsilon$ and $t_{0}>1$ such that%
\[
tw\left(  x,t\right)  >\left(  C+\epsilon \right)  d
\]
for $t>t_{0}$. Here we have utilized the monotonicity of $tw\left(
x,t\right)  $. Therefore, by integrating both sides, we have%
\[
v\left(  x,t\right)  \geq \left(  C+\epsilon \right)  d\log t-A
\]
where the constant $A$ is independent of $t$. But this contradicts the
inequality $\left(  \ref{321}\right)  $. So we obtain $\left(  \ref{317}%
\right)  $. As a result, we deduce $\left(  \ref{318}\right)  $.

Now, other than fixing the constant $C$ chosen in $\left(  \ref{318}\right)
$, as well as in $\left(  \ref{321}\right)  $, we may assume such constant is
a positive integer. Next we want to settle the dimension estimate $\left(
\ref{315}\right)  $. In spite of the proof of $\left(  \ref{315}\right)  $ is
the same as the one in \cite{m1}, we would write it down for completeness. Let
$k\left(  n\right)  $ be a constant satisfying%
\[
q\left(  m\right)  =\left(
\begin{array}
[c]{c}%
n+m\\
n
\end{array}
\right)  <k\left(  n\right)  m^{n}.
\]
Consider the map%
\[%
\begin{array}
[c]{c}%
\Phi:\mathcal{O}_{d}^{CR}\left(  M\right)  \longrightarrow%
\mathbb{C}
^{q\left(  Cd\right)  }\\
\text{ \  \  \  \  \  \  \  \  \  \  \  \  \  \  \  \  \  \  \  \  \  \  \  \  \ }f\longmapsto \left(
D^{\alpha}f\right)  _{\left \vert \alpha \right \vert \leq Cd}.
\end{array}
\]
From $\left(  \ref{318}\right)  $, we see that $\Phi$ is injective. Suppose%
\[
\dim_{%
\mathbb{C}
}\left(  \mathcal{O}_{d}^{CR}\left(  M\right)  \right)  >C^{\prime}d^{n}%
\]
where $C^{\prime}$ is chosen with $C^{\prime}>k\left(  n\right)  C^{n}$, this
implies that%
\[
\dim_{%
\mathbb{C}
}\left(  \mathcal{O}_{d}^{CR}\left(  M\right)  \right)  >k\left(  n\right)
(Cd)^{n}>q\left(  Cd\right)  .
\]
However, this contradicts with the fact that $\Phi$ is injective. We complete
the proof.
\end{proof}

\section{Sharp Dimension Estimates}

Subsequently, we give another lemma also substantial to the proof of the sharp
dimension estimate. In contrast with Lemma \ref{l43}, one could find the
following lemma is the stronger version of that. In fact, that is why we could
derive the more advanced result-the sharp dimension estimate than the one in
the last section.

\begin{lemma}
\label{51S}Let $(M,J,\theta)$ be a complete noncompact pseudohermitian
$(2n+1)$-manifold of vanishing torsion with nonnegative pseudohermitian
bisectional curvature. Assume that $v\left(  x,t\right)  $ is a solution to
the CR heat equation%
\[
\left(  \frac{\partial}{\partial t}-\Delta_{b}\right)  v\left(  x,t\right)
=0
\]
on $M\times \left(  0,+\infty \right)  $ with the initial condition%
\[
v\left(  x,0\right)  =2\log \left \vert f\left(  x\right)  \right \vert
\]
for $f\in \mathcal{O}_{d}^{CR}\left(  M\right)  $, and
\[
w\left(  x,t\right)  =\frac{\partial}{\partial t}v\left(  x,t\right)  .
\]
Then%
\begin{equation}
\lim_{t\rightarrow0^{+}}tw\left(  x,t\right)  =ord_{x}f. \label{328}%
\end{equation}

\end{lemma}

\begin{proof}
From the definition of $w\left(  x,t\right)  $, we have
\begin{equation}
w\left(  x,t\right)  =2%
{\displaystyle \int \limits_{M}}
H(x,y,t)\Delta_{b}\log \left \vert f\left(  y\right)  \right \vert dy.
\label{327}%
\end{equation}

By the CR heat kernel estimate, the same computations of $\left(
\ref{322}\right)  $ in Lemma \ref{l43} give us the inequality
\begin{equation}
\frac{1}{V\left(  B_{x}\left(  r\right)  \right)  }%
{\displaystyle \int \limits_{B_{x}\left(  r\right)  }}
\Delta_{b}\log \left \vert f\left(  y\right)  \right \vert dy\leq Cw\left(
x,r^{2}\right)  . \label{325}%
\end{equation}

Combing $\left(  \ref{325}\right)  $ and $\left(  \ref{317}\right)  $, we
obtain%
\begin{equation}
\frac{1}{V\left(  B_{x}\left(  r\right)  \right)  }%
{\displaystyle \int \limits_{B_{x}\left(  r\right)  }}
\Delta_{b}\log \left \vert f\left(  y\right)  \right \vert dy\leq C\frac{d}%
{r^{2}} \label{326}%
\end{equation}
for any $r>0$ by $\left(  \ref{324}\right)  $. By the equality%
\[
ord_{\widetilde{x}}\left(  \widetilde{f}\right)  =\frac{1}{2n}\lim
_{r\rightarrow0^{+}}\frac{r^{2}}{V\left(  B_{\widetilde{x}}\left(  r\right)
\right)  }%
{\displaystyle \int \limits_{B_{\widetilde{x}}\left(  r\right)  }}
\widetilde{\Delta}\log \left \vert \widetilde{f}\right \vert ,
\]
we know that for any $\epsilon>0,$ there is $\delta>0$ such that%
\begin{equation}
\left \vert 2nord_{\widetilde{x}}\left(  \widetilde{f}\right)  -\frac{r^{2}%
}{V\left(  B_{\widetilde{x}}\left(  r\right)  \right)  }%
{\displaystyle \int \limits_{B_{\widetilde{x}}\left(  r\right)  }}
\widetilde{\Delta}\log \left \vert \widetilde{f}\right \vert \right \vert
<\frac{\epsilon}{6} \label{336}%
\end{equation}
for $0<r<\delta<1$. Now we separate the integration in $\left(  \ref{327}%
\right)  $ into two parts:%
\begin{equation}%
\begin{array}
[c]{cl}
& tw\left(  x,t\right) \\
= & 2t%
{\displaystyle \int \limits_{M}}
H(x,y,t)\Delta_{b}\log \left \vert f\left(  y\right)  \right \vert dy\\
= & 2t[%
{\displaystyle \int \limits_{B_{x}\left(  \delta \right)  ^{c}}}
H(x,y,t)\Delta_{b}\log \left \vert f\left(  y\right)  \right \vert dy+%
{\displaystyle \int \limits_{B_{x}\left(  \delta \right)  }}
H(x,y,t)\Delta_{b}\log \left \vert f\left(  y\right)  \right \vert dy].
\end{array}
\label{329}%
\end{equation}

So if we could show that (I) the first integral is close to zero and (II) the
second one is close to $ord_{x}\left(  f\right)  $ as $t$ $\rightarrow0^{+}$
in the last line of $\left(  \ref{329}\right)  $, then $\left(  \ref{328}%
\right)  $ holds.

(I). Now we're going to show the first integral in $\left(  \ref{329}\right)
$ goes to zero for sufficiently small $t$ as follows: \ By CR heat kernel
estimate, we have, for $t\leq \delta^{2}$,%
\begin{equation}%
\begin{array}
[c]{cl}
& 2t%
{\displaystyle \int \limits_{B_{x}\left(  \delta \right)  ^{c}}}
H(x,y,t)\Delta_{b}\log \left \vert f\left(  y\right)  \right \vert dy\\
\leq & 2t\frac{C}{V\left(  B_{x}\left(  \sqrt{t}\right)  \right)  }%
\int_{\delta}^{+\infty}\exp(-C\frac{r^{2}}{2t})(%
{\displaystyle \int \limits_{\partial B_{x}\left(  r\right)  }}
\Delta_{b}\log \left \vert f\left(  s\right)  \right \vert d\sigma \left(
s\right)  )dr\\
\leq & -\frac{Ct}{V\left(  B_{x}\left(  \sqrt{t}\right)  \right)  }%
\exp(-C\frac{\delta^{2}}{2t})%
{\displaystyle \int \limits_{B_{x}\left(  \delta \right)  }}
\Delta_{b}\log \left \vert f\left(  y\right)  \right \vert dy\\
& +Ct\int_{\delta}^{+\infty}[\exp(-C\frac{r^{2}}{2t})\frac{r}{t}(\frac
{r}{\sqrt{t}})^{2C}(\frac{1}{V\left(  B_{x}\left(  r\right)  \right)  }%
{\displaystyle \int \limits_{B_{x}\left(  r\right)  }}
\Delta_{b}\log \left \vert f\left(  y\right)  \right \vert dy)]dr.
\end{array}
\label{330}%
\end{equation}
Here we use integration by part, $\left(  \ref{326}\right)  $ and the CR
volume doubling property in the second inequality. From the CR volume doubling
property and the inequality $\left(  \ref{326}\right)  $, we obtain%
\begin{equation}
\lim_{t\rightarrow0^{+}}\frac{C}{V\left(  B_{x}\left(  \sqrt{t}\right)
\right)  }\exp(-C\frac{\delta^{2}}{2t})%
{\displaystyle \int \limits_{B_{x}\left(  \delta \right)  }}
\Delta_{b}\log \left \vert f\left(  y\right)  \right \vert dy=0. \label{331}%
\end{equation}

Set $\tau=\frac{r^{2}}{2t}$ in the second integral in the right hand side in
$\left(  \ref{330}\right)  $, we have%
\begin{equation}%
\begin{array}
[c]{cl}
& C\int_{\delta}^{+\infty}[\exp(-C\frac{r^{2}}{2t})r(\frac{r}{\sqrt{t}%
})^{2C_{9}}(\frac{1}{V\left(  B_{x}\left(  r\right)  \right)  }%
{\displaystyle \int \limits_{B_{x}\left(  r\right)  }}
\Delta_{b}\log \left \vert f\left(  y\right)  \right \vert dy)]dr\\
\leq & C\int_{\frac{\delta^{2}}{2t}}^{+\infty}[\exp \left(  -C\tau \right)
\left(  \tau t\right)  \left(  2\tau \right)  ^{C_{9}-1}(\frac{1}{V\left(
B_{x}\left(  \sqrt{2\tau t}\right)  \right)  }%
{\displaystyle \int \limits_{B_{x}\left(  \sqrt{2\tau t}\right)  }}
\Delta_{b}\log \left \vert f\left(  y\right)  \right \vert dy)]d\tau \\
\leq & Cd\int_{\frac{\delta^{2}}{2t}}^{+\infty}\exp \left(  -C\tau \right)
\left(  2\tau \right)  ^{C_{9}-1}d\tau \\
\rightarrow & 0^{+}\text{ \ }as\text{ }t\rightarrow0^{+}%
\end{array}
\text{ } \label{332}%
\end{equation}
by $\left(  \ref{326}\right)  $. From $\left(  \ref{330}\right)  $, $\left(
\ref{331}\right)  $ and $\left(  \ref{332}\right)  $, we see that the first
integration in the last line of $\left(  \ref{329}\right)  $ is close to zero,
as $t$ goes to zero.

(II). Next we show the following two inequalities from which the fact that the
second integral in the last line of $\left(  \ref{329}\right)  $ approximates
to $ord_{x}\left(  f\right)  $ as $t$ small enough, could be deduced :%
\begin{equation}
ord_{x}\left(  f\right)  -\epsilon<\lim \inf_{t\rightarrow0^{+}}(2t%
{\displaystyle \int \limits_{B_{x}\left(  \delta \right)  }}
H(x,y,t)\Delta_{b}\log \left \vert f\left(  y\right)  \right \vert dy)
\label{333}%
\end{equation}
and%
\begin{equation}
\lim \sup_{t\rightarrow0^{+}}(2t%
{\displaystyle \int \limits_{B_{x}\left(  \delta \right)  }}
H(x,y,t)\Delta_{b}\log \left \vert f\left(  y\right)  \right \vert dy)<ord_{x}%
\left(  f\right)  +\epsilon. \label{334}%
\end{equation}

Now we prove the inequality $\left(  \ref{333}\right)  $. As for $\left(
\ref{334}\right)  $, the deduction is almost the same as well as $\left(
\ref{333}\right)  $.

In the rest of the proof, we adopt the adapted metric as in (\ref{2017}). From
the small time asymptotic expansion for the heat kernel $H_{\epsilon}\left(
x,y,t\right)  $ with respect to the adapted metric,
\[
H_{\epsilon}\left(  x,y,t\right)  \sim \frac{1}{\left(  4\pi t\right)
^{\frac{2n+1}{2}}}\exp(-\frac{d^{2}\left(  x,y\right)  }{4t})+L.O.T.
\]
if $t$ is sufficiently small, we have, for $t<<1$,%
\begin{equation}%
\begin{array}
[c]{cl}
& 2t%
{\displaystyle \int \limits_{B_{x}\left(  \delta \right)  }}
H(x,y,t)\Delta_{b}\log \left \vert f\left(  y\right)  \right \vert dy\\
= & t%
{\displaystyle \int \limits_{B_{x}\left(  \delta \right)  }}
H_{\epsilon}(x,y,\frac{t}{2})\Delta_{\epsilon}\log \left \vert f\left(
y\right)  \right \vert d\mu_{\epsilon}\left(  y\right) \\
\geq & t\int_{0}^{\delta}[\frac{1}{\left(  2\pi t\right)  ^{\frac{2n+1}{2}}%
}\exp \left(  -\frac{s^{2}}{2t}\right)  (%
{\displaystyle \int \limits_{\partial B_{x}\left(  s\right)  }}
\Delta_{\epsilon}\log \left \vert f\left(  y\right)  \right \vert d\sigma
_{\epsilon}\left(  y\right)  )]ds-\frac{\epsilon}{2}\\
= & t\frac{1}{\left(  2\pi t\right)  ^{\frac{2n+1}{2}}}\exp \left(
-\frac{\delta^{2}}{2t}\right)  (%
{\displaystyle \int \limits_{B_{x}\left(  \delta \right)  }}
\Delta_{\epsilon}\log \left \vert f\left(  y\right)  \right \vert d\mu_{\epsilon
}\left(  y\right)  )\\
& +t\int_{0}^{\delta}[\frac{1}{\left(  2\pi t\right)  ^{\frac{2n+1}{2}}}%
\exp \left(  -\frac{s^{2}}{2t}\right)  \frac{s}{t}(%
{\displaystyle \int \limits_{B_{x}\left(  s\right)  }}
\Delta_{\epsilon}\log \left \vert f\left(  y\right)  \right \vert d\mu_{\epsilon
}\left(  y\right)  )]ds-\frac{\epsilon}{2}\\
= & t\frac{1}{\left(  2\pi t\right)  ^{\frac{2n+1}{2}}}\exp \left(
-\frac{\delta^{2}}{2t}\right)  (%
{\displaystyle \int \limits_{B_{x}\left(  \delta \right)  }}
\Delta_{b}\log \left \vert f\left(  y\right)  \right \vert d\mu \left(  y\right)
)\\
& +t\int_{0}^{\delta}[\frac{1}{\left(  2\pi t\right)  ^{\frac{2n+1}{2}}}%
\exp \left(  -\frac{s^{2}}{2t}\right)  \frac{s}{t}(%
{\displaystyle \int \limits_{B_{x}\left(  s\right)  }}
\Delta_{b}\log \left \vert f\left(  y\right)  \right \vert d\mu \left(  y\right)
)]ds-\frac{\epsilon}{2}.
\end{array}
\label{335}%
\end{equation}
Here we use the fact that $f_{0}=0$ and use integration by part, $\left(
\ref{326}\right)  $ and%
\[
\left \{
\begin{array}
[c]{l}%
\Delta_{\epsilon}\log \left \vert f\left(  y\right)  \right \vert =2\Delta
_{b}\log \left \vert f\left(  y\right)  \right \vert ,\\
d\mu_{\epsilon}\left(  y\right)  =\frac{1}{2}d\mu.
\end{array}
\right.  .
\]
More precisely, we explain why the first equality in $\left(  \ref{335}%
\right)  $ holds as follows : \ On $B_{x}\left(  \delta \right)  $, the
equation
\begin{equation}
\left \{
\begin{array}
[c]{l}%
\left(  \frac{\partial}{\partial t}-\Delta_{b}\right)  w\left(  x,t\right)
=0\\
w\left(  x,0\right)  =2\Delta_{b}\log \left \vert f\left(  x\right)
\right \vert
\end{array}
\right.  \label{344}%
\end{equation}
admits the solution%
\[
w\left(  x,t\right)  =2%
{\displaystyle \int \limits_{B_{x}\left(  \delta \right)  }}
H(x,y,t)\Delta_{b}\log \left \vert f\left(  y\right)  \right \vert dy.
\]

Differentiating the equation $\left(  \ref{344}\right)  $ along the direction
$T$, we have, by $f_{0}=0$,
\[
\left \{
\begin{array}
[c]{l}%
\left(  \frac{\partial}{\partial t}-\Delta_{b}\right)  w_{0}\left(
x,t\right)  =0,\\
w_{0}\left(  x,0\right)  =0.
\end{array}
\right.
\]

By Theorem \ref{T31}, we know that%
\[
w_{0}\left(  x,t\right)  =0.
\]

On the other hand, it follows from the equality $\Delta_{\epsilon}=2\Delta
_{b}+\epsilon^{2}T^{2}$ that the equation $\left(  \ref{344}\right)  $ becomes%
\[
\left \{
\begin{array}
[c]{l}%
\left(  \frac{\partial}{\partial t}-\frac{1}{2}\Delta_{\epsilon}\right)
w\left(  x,t\right)  =0,\\
w\left(  x,0\right)  =2\Delta_{b}\log \left \vert f\left(  x\right)  \right \vert
.
\end{array}
\right.  .
\]

It could ne written in this form%
\[
\left \{
\begin{array}
[c]{l}%
\left(  \frac{\partial}{\partial \widetilde{t}}-\Delta_{\epsilon}\right)
\widetilde{w}\left(  x,t\right)  =0,\\
\widetilde{w}\left(  x,0\right)  =\Delta_{\epsilon}\log \left \vert f\left(
x\right)  \right \vert .
\end{array}
\right.
\]

where $\widetilde{w}\left(  x,t\right)  =w\left(  x,2t\right)  $. From the
result in \cite{do}, we obtain that such equation admits a unique solution%
\[
\widetilde{w}\left(  x,t\right)  =%
{\displaystyle \int \limits_{B_{x}\left(  \delta \right)  }}
H_{\epsilon}(x,y,t)\Delta_{\epsilon}\log \left \vert f\left(  y\right)
\right \vert d\mu_{\epsilon}\left(  y\right)  ,
\]

that is,%
\[
w\left(  x,2t\right)  =%
{\displaystyle \int \limits_{B_{x}\left(  \delta \right)  }}
H_{\epsilon}(x,y,t)\Delta_{\epsilon}\log \left \vert f\left(  y\right)
\right \vert d\mu_{\epsilon}\left(  y\right)  .
\]

By the uniqueness, we have%
\[
2%
{\displaystyle \int \limits_{B_{x}\left(  \delta \right)  }}
H(x,y,t)\Delta_{b}\log \left \vert f\left(  y\right)  \right \vert dy=%
{\displaystyle \int \limits_{B_{x}\left(  \delta \right)  }}
H_{\epsilon}(x,y,\frac{t}{2})\Delta_{\epsilon}\log \left \vert f\left(
y\right)  \right \vert d\mu_{\epsilon}\left(  y\right)  .
\]

By the inequality $\left(  \ref{326}\right)  $ again, we also have%
\[%
\begin{array}
[c]{cl}
& \underset{t\rightarrow0^{+}}{\lim}t\frac{1}{\left(  2\pi t\right)
^{\frac{2n+1}{2}}}\exp(-\frac{\delta^{2}}{2t})(%
{\displaystyle \int \limits_{B_{x}\left(  \delta \right)  }}
\Delta_{b}\log \left \vert f\left(  y\right)  \right \vert d\mu \left(  y\right)
)\\
\leq & \underset{t\rightarrow0^{+}}{\lim}t\frac{1}{\left(  2\pi t\right)
^{\frac{2n+1}{2}}}\exp(-\frac{\delta^{2}}{2t})V\left(  B_{x}\left(
\delta \right)  \right)  C\frac{d}{\delta^{2}}\\
= & 0.
\end{array}
\]

From the result above, we simplify the inequality $\left(  \ref{335}\right)  $
into
\[%
\begin{array}
[c]{cl}
& t%
{\displaystyle \int \limits_{B_{x}\left(  \delta \right)  }}
H(x,y,t)\Delta_{b}\log \left \vert f\left(  y\right)  \right \vert dy\\
\geq & t\int_{0}^{\delta}\frac{1}{\left(  2\pi t\right)  ^{\frac{2n+1}{2}}%
}\exp(-\frac{s^{2}}{2t})\frac{s}{t}(%
{\displaystyle \int \limits_{B_{x}\left(  s\right)  }}
\Delta_{b}\log \left \vert f\left(  y\right)  \right \vert dy)ds-\frac{3\epsilon
}{4}%
\end{array}
\]
for $t<<1$. Hence we just need to claim that%
\begin{equation}
t\int_{0}^{\delta}[\frac{1}{\left(  2\pi t\right)  ^{\frac{2n+1}{2}}}%
\exp(-\frac{s^{2}}{2t})\frac{s}{t}(%
{\displaystyle \int \limits_{B_{x}\left(  s\right)  }}
\Delta_{b}\log \left \vert f\left(  y\right)  \right \vert dy)]ds>ord_{x}%
f-\frac{\epsilon}{4} \label{337}%
\end{equation}
for $t<<1$. By the hypotheses, we see that%
\begin{equation}
V\left(  B_{\widetilde{x}}\left(  r\right)  \right)  \sim \omega_{2n}r^{2n}
\label{338}%
\end{equation}
as $r\rightarrow0^{+}$, where $\omega_{k}$ is the volume of the unit ball in $%
\mathbb{R}
^{k}$ for all $k\in%
\mathbb{N}
$. Utilizing this volume approximation and $\left(  \ref{336}\right)  $, we
obtain, for $t<<1$,
\begin{equation}%
\begin{array}
[c]{cl}
& t\int_{0}^{\delta}[\frac{1}{\left(  2\pi t\right)  ^{\frac{2n+1}{2}}}%
\exp(-\frac{s^{2}}{2t})\frac{s}{t}(%
{\displaystyle \int \limits_{B_{x}\left(  s\right)  }}
\Delta_{b}\log \left \vert f\left(  y\right)  \right \vert dy)]ds\\
= & t\int_{0}^{\delta}\frac{1}{\left(  2\pi t\right)  ^{\frac{2n+1}{2}}}%
\exp(-\frac{s^{2}}{2t})\frac{s}{t}[\int_{-s}^{s}(%
{\displaystyle \int \limits_{B_{\widetilde{x}}\left(  \sqrt{s^{2}-l^{2}}\right)
}}
\Delta_{\widetilde{y}}\log \left \vert f\left(  \left(  \widetilde{y},t\right)
\right)  \right \vert d\widetilde{y})dl]ds\\
> & \int_{0}^{\delta}[\frac{1}{\left(  2\pi t\right)  ^{\frac{2n+1}{2}}}%
\exp(-\frac{s^{2}}{2t})2ns\left(  ord_{x}\left(  f\right)  \right)
\omega_{2n}(\int_{-s}^{s}\left(  s^{2}-l^{2}\right)  ^{n-1}dl)]ds-\frac
{\epsilon}{6}\\
= & \int_{0}^{\delta}[\frac{1}{\left(  2\pi t\right)  ^{\frac{2n+1}{2}}}%
\exp(-\frac{s^{2}}{2t})2ns\left(  ord_{x}\left(  f\right)  \right)
\omega_{2n}(s^{2n-1}\sqrt{\pi}\frac{\Gamma \left(  n\right)  }{\Gamma \left(
n+\frac{1}{2}\right)  })]ds-\frac{\epsilon}{6}\\
= & \frac{2nt}{\left(  2\pi t\right)  ^{\frac{2n+1}{2}}}ord_{x}\left(
f\right)  \omega_{2n}\sqrt{\pi}\frac{\Gamma \left(  n\right)  }{\Gamma \left(
n+\frac{1}{2}\right)  }(\int_{0}^{\frac{\delta^{2}}{2t}}\exp \left(
-\tau \right)  \left(  2\tau t\right)  ^{\frac{2n-1}{2}}d\tau)-\frac{\epsilon
}{6}\\
= & \frac{n\omega_{2n}ord_{x}\left(  f\right)  }{\pi^{n}}\frac{\Gamma \left(
n\right)  }{\Gamma \left(  n+\frac{1}{2}\right)  }\int_{0}^{\frac{\delta^{2}%
}{2t}}\exp \left(  -\tau \right)  \tau^{\frac{2n-1}{2}}d\tau-\frac{\epsilon}%
{6}\\
> & ord_{x}\left(  f\right)  -\frac{\epsilon}{4}%
\end{array}
\label{341}%
\end{equation}
where we use the error term $\frac{\epsilon}{12}$ twice in the third line,
that are $\left(  \ref{336}\right)  $ and $\left(  \ref{338}\right)  $, and
the fact that%
\[
\int_{0}^{\frac{\delta^{2}}{2t}}\exp \left(  -\tau \right)  \tau^{\frac{2n-1}%
{2}}d\tau \rightarrow \Gamma(n+\frac{1}{2})
\]
as $t$ goes to zero, in the last inequality. For the second equality of
$\left(  \ref{341}\right)  $, we use the equality $\left(  \ref{343}\right)
$. This completes the proof.
\end{proof}

Finally, we would show the main theorem-the sharp dimension estimate in this
paper as follows.

\begin{proof}
of \textbf{Theorem \ref{TA} \ }

Here we follow the ideas and arguments in \cite{n} and \cite{cfyz} to derive
this result and use the notations of the hypotheses in the preceding lemma.

From%
\[
\frac{\partial}{\partial t}\left(  tw\left(  x,t\right)  \right)  \geq0,
\]
we have, for the positive integer $k\geq2$,%
\[%
\begin{array}
[c]{ccl}%
v\left(  x,t\right)  & = & (%
{\displaystyle \int \limits_{\left(  1,t\right)  }}
w\left(  x,s\right)  ds)+v\left(  x,1\right) \\
& \geq & (%
{\displaystyle \int \limits_{\left(  t^{\frac{1}{k}},t\right)  }}
\frac{t^{\frac{1}{k}}}{s}w\left(  x,t^{\frac{1}{k}}\right)  ds)+v\left(
x,1\right) \\
& = & t^{\frac{1}{k}}w(x,t^{\frac{1}{k}})(1-\frac{1}{k})\log t+v\left(
x,1\right)  .
\end{array}
\]
Then, by Lemma \ref{l4.2}, we obtain%
\[
\underset{t\rightarrow+\infty}{\lim \sup}[t^{\frac{1}{k}}w(x,t^{\frac{1}{k}%
})(1-\frac{1}{k})]\leq d.
\]

On account of the arbitrariness of $k$, we have%
\[
\underset{t\rightarrow+\infty}{\lim \sup}tw\left(  x,t\right)  \leq d.
\]

By the monotonicity of $tw\left(  x,t\right)  $ and Lemma \ref{51S}, we get%
\begin{equation}
ord_{x}f\leq d \label{340}%
\end{equation}
for all nontrivial $f\in \mathcal{O}_{d}^{CR}\left(  M\right)  $. Now we
perform the Poincar\'{e}-Siegel argument and note that the argument only
depends on the local property of the CR holomorphic functions of polynomial
growth. The Poincar\'{e}-Siegel map%
\[%
\begin{array}
[c]{c}%
\Phi:\mathcal{O}_{d}^{CR}\left(  M\right)  \longrightarrow%
\mathbb{C}
^{q\left(  d\right)  }\\
\text{ \  \  \  \  \  \  \  \  \  \  \  \  \  \  \  \  \  \  \  \  \  \  \  \  \ }f\longmapsto \left(
D^{\alpha}f\right)  _{\left \vert \alpha \right \vert \leq d},
\end{array}
\]
where $q\left(  d\right)  =\left(
\begin{array}
[c]{c}%
n+d\\
n
\end{array}
\right)  $ is the complex dimension of $\mathcal{O}_{d}\left(
\mathbb{C}
^{n}\right)  $. So if $f\neq0$ but $D^{\alpha}f=0$ for all $\alpha$ with
$\left \vert \alpha \right \vert \leq d$, then it contradicts with $\left(
\ref{340}\right)  $. Hence, we know that the Poincar\'{e}-Siegel map $\Phi$ is
injective and this implies that%
\[
\dim_{%
\mathbb{C}
}\left(  \mathcal{O}_{d}^{CR}\left(  M\right)  \right)  \leq \dim_{%
\mathbb{C}
}\mathcal{O}_{d}\left(
\mathbb{C}
^{n}\right)  =\dim_{%
\mathbb{C}
}\left(  \mathcal{O}_{d}^{CR}\left(  H^{n}\right)  \right)  .
\]

Therefore we complete the proof.
\end{proof}

\appendix

\section{}

In this appendix, we would show the CR $L^{p}$ submean value inequality.
Although the details are almost the same as in \cite{sc2}, we give its proof
here for completeness.

\begin{proof}
of \textbf{Proposition \ref{P31}} : Initially we prove the case $p=2$ and may
assume $\tau=1=\delta^{\prime}$. On account of
\[%
{\displaystyle \int \limits_{B_{x}(r)}}
\langle(\frac{\partial}{\partial t}-\triangle_{b})u,\phi \rangle d\mu \leq0
\]
for any $\phi \in C_{c}^{\infty}(B_{x}(r))$, choosing $\phi=\psi^{2}u$ for
$\psi \in C_{c}^{\infty}(B_{x}(r))$, we have
\[%
\begin{array}
[c]{ccl}%
{\displaystyle \int \limits_{B_{x}(r)}}
[\psi^{2}u\frac{\partial u}{\partial t}+\psi^{2}|\nabla_{b}u|^{2}]d\mu & \leq
& 2|%
{\displaystyle \int \limits_{B_{x}(r)}}
u\psi \langle \nabla_{b}u,\nabla_{b}\psi \rangle d\mu|\\
& \leq & 2\
{\displaystyle \int \limits_{B_{x}(r)}}
|\nabla_{b}\psi|^{2}u^{2}d\mu+\frac{1}{2}%
{\displaystyle \int \limits_{B_{x}(r)}}
\psi^{2}|\nabla_{b}u|^{2}d\mu.
\end{array}
\]
That is,
\[%
{\displaystyle \int \limits_{B_{x}(r)}}
(2\psi^{2}u\frac{\partial u}{\partial t}+\psi^{2}|\nabla_{b}u|^{2})d\mu
\leq4\int_{B_{x}(r)}|\nabla_{b}\psi|^{2}u^{2}d\mu.
\]
So there exists a positive constant $C>0$ such that,
\[%
{\displaystyle \int \limits_{B_{x}(r)}}
(2\psi^{2}u\frac{\partial u}{\partial t}+\psi^{2}|\nabla_{b}u|^{2})d\mu \leq
C||\nabla_{b}\psi||_{L^{\infty}(B_{x}(r))}^{2}%
{\displaystyle \int \limits_{supp(\psi)}}
u^{2}d\mu.
\]
Hereafter we adopt the notations $C$ and $A$ as constants which may be
different in each line.

Let $\chi(t)\in C^{\infty}(R)$ be a smooth function on $R$. We have%
\begin{equation}%
\begin{array}
[c]{cl}
& \frac{\partial}{\partial t}%
{\displaystyle \int \limits_{B_{x}(r)}}
(\chi \psi u)^{2}d\mu+\chi^{2}%
{\displaystyle \int \limits_{B_{x}(r)}}
|\nabla_{b}(\psi u)|^{2}d\mu \\
\leq & C\chi(\chi||\nabla_{b}\psi||_{L^{\infty}(B_{x}(r))}^{2}+\left \Vert
\psi \right \Vert _{L^{\infty}\left(  B_{x}(r)\right)  }^{2}||\chi^{\prime
}||_{L^{\infty}(R)}^{2})\int_{supp\psi}u^{2}d\mu.
\end{array}
\label{306}%
\end{equation}
Now we choose the appropriate the smooth functions $\psi$ and $\chi$ as
follows : Let $0<\sigma^{\prime}<\sigma<1$ and $\omega=\left(  \sigma
-\sigma^{\prime}\right)  $, we define
\[
\psi(y)=\left \{
\begin{array}
[c]{l}%
1,\quad on\text{ }B_{x}(\sigma^{\prime}r),\\
0,\quad on\text{ }M\backslash B_{x}(\sigma r),
\end{array}
\right.
\]
with $0\leq \psi \leq1$ and $|\nabla_{b}\psi|\leq \frac{2}{\omega r}$. Define
\[
\chi(t)=\left \{
\begin{array}
[c]{l}%
0,\quad on\text{ }(-\infty,s-\sigma r^{2}),\\
1,\quad on\text{ }(s-\sigma^{\prime}r^{2},\infty),
\end{array}
\right.
\]
with $0\leq \chi \leq1$ and $|\chi^{\prime}|\leq \frac{2}{\omega r^{2}}$. \ Let
$I_{\sigma}=(s-\sigma r^{2},s)$. Integrating (\ref{306}) over $(s-r^{2},t)$
for $t\in I_{\sigma^{\prime}}$, we have
\begin{equation}
\sup_{I_{\sigma^{\prime}}}\{ \int_{B_{x}(r)}(\psi u)^{2}d\mu \}+\int
_{B_{x}(r)\times I_{\sigma^{\prime}}}|\nabla_{b}(\psi u)|^{2}d\overline{\mu
}\leq \frac{C}{\left(  \omega r\right)  ^{2}}\int_{Q_{\sigma}}u^{2}%
d\overline{\mu}. \label{307}%
\end{equation}

Let $k(r,\nu)=\frac{Cr^{2}}{V(B_{x}(r))^{\frac{2}{\nu}}}$. By H\"{o}lder's
inequality and CR Sobolev inequality (\ref{p31}), we have%
\[%
\begin{array}
[c]{ccl}%
\int_{B_{x}(\sigma^{\prime}r)}(\psi u)^{2(1+\frac{2}{\nu})}d\mu & \leq &
\int_{B_{x}(r)}(\psi u)^{2(1+\frac{2}{\nu})}d\mu \\
& = & \int_{B_{x}(r)}(\psi u)^{2}(\psi u)^{\frac{4}{\nu}}d\mu \\
& \leq & (\int_{B_{x}(r)}(\psi u)^{\frac{2\nu}{\nu-2}}d\mu)^{\frac{\nu-2}{\nu
}}(\int_{B_{x}(r)}(\psi u)^{2}d\mu)^{\frac{2}{\nu}}\\
& \leq & [k\int_{B_{x}(r)}(|\nabla_{b}(\psi u))|^{2}+|\frac{\psi u}{r}%
|^{2})d\mu]||\psi u||_{L^{2}(B_{x}(r))}^{\frac{4}{\nu}}.
\end{array}
\]
Then, integrating over $I_{\sigma^{\prime}}$,
\[
\int_{Q_{\sigma^{^{\prime}}}}u^{2(1+\frac{2}{\nu})}d\overline{\mu}\leq
k(\frac{A}{(r\omega)^{2}}\int_{Q_{\sigma}}u^{2}d\overline{\mu})^{1+\frac
{2}{\nu}}\quad
\]
by utilizing $\left(  \ref{307}\right)  $ twice. Let $\theta=\left(
1+\frac{2}{\nu}\right)  $ and we note that $u^{p}$ is also a positive smooth
subsolution to the CR heat equation if $u$ is a positive smooth subsolution to
the CR heat equation for $p\geq1$. So we have the following inequality
\begin{equation}
\int_{Q_{\sigma^{^{\prime}}}}u^{2p\theta}d\overline{\mu}\leq k(\frac
{A}{(r\omega)^{2}}\int_{Q_{\sigma}}u^{2p}d\overline{\mu})^{\theta}.
\label{308}%
\end{equation}

Let $\omega_{j}=\frac{(1-\delta)}{2^{j}}$ and $\sigma_{0}=1$ and $\sigma
_{k+1}=1-%
{\displaystyle \sum \limits_{j\in I_{k}}}
\omega_{j}$ for $k\geq0$. Apply $\left(  \ref{308}\right)  $ with
$p=\theta^{i}$, $\sigma=\sigma_{i}$ and $\sigma^{\prime}=\sigma_{i+1}$, we
have
\begin{equation}
\int_{Q_{\sigma_{i+1}}}u^{2\theta^{i+1}}d\overline{\mu}\leq k(\frac
{A}{(r\omega_{i})^{2}}\int_{Q_{\sigma_{i}}}u^{2\theta^{i}}d\overline{\mu
})^{\theta}. \label{309}%
\end{equation}
By iteration, we obtain
\[
(\int_{Q_{\sigma_{i+1}}}u^{2\theta^{i+1}})^{\theta^{-1-i}}\leq A(\nu)k^{(%
{\displaystyle \sum \limits_{l\in I_{i}^{^{\prime}}}}
\theta^{l-1-i})}\left(  (1-\delta)r\right)  ^{(-2%
{\displaystyle \sum \limits_{l\in I_{i}^{^{\prime}}}}
\theta^{l-i})}(\int_{Q_{0}}u^{2}d\overline{\mu}).
\]

Letting $i$ tend to $\infty$, we have
\begin{equation}
\sup_{Q_{\delta}}\{u^{2}\} \leq A(\nu)k^{\frac{\nu}{2}}((1-\delta)r)^{-2-\nu
}\int_{Q}u^{2}d\overline{\mu}. \label{310}%
\end{equation}

This completes the case $p=2$. When $p>2$, then the transformation $u\mapsto
u^{\frac{p}{2}}$ still satisfies the hypothesis for $p=2$. Hence, it is easy
to validate this case. As for $p\in(0,2)$, the case $p=2$ yields that%
\[
\sup_{\widetilde{Q}_{\sigma}}\left \{  u\right \}  \leq \frac{\widetilde
{k}\left(  \tau,\nu,r\right)  }{(\rho-\sigma)^{1+\frac{\nu}{2}}}(%
{\displaystyle \int \limits_{\widetilde{Q}_{\rho}}}
u^{2}d\overline{\nu})^{\frac{1}{2}},
\]
where $\widetilde{k}\left(  \tau,\nu,r\right)  =\frac{A\left(  \tau
,\nu,2\right)  C^{\frac{\nu}{2}}}{r^{2}},$ $d\overline{\nu}=d\nu dt=$ the
volume-normalized measure $\frac{d\overline{\mu}}{V\left(  \widetilde
{Q}\right)  },$ and $\widetilde{Q}=Q_{\delta^{\prime}},$ $\widetilde
{Q}_{\sigma}=\left(  \widetilde{Q}\right)  _{\sigma},$ $\widetilde{Q}_{\rho
}=\left(  \widetilde{Q}\right)  _{\rho},$ $\sigma \in \left(  \frac{1}%
{2},1\right)  ,$ $\left(  1-\rho \right)  =\frac{3}{4}\left(  1-\sigma \right)
.$ Due to the inequality%
\[
\left \Vert u\right \Vert _{L^{2}\left(  \widetilde{Q}_{\rho},\overline{\nu
}\right)  }\leq \left \Vert u\right \Vert _{L^{\infty}\left(  \widetilde{Q}%
_{\rho},\overline{\nu}\right)  }^{1-\frac{p}{2}}\left \Vert u\right \Vert
_{L^{p}\left(  \widetilde{Q},\overline{\nu}\right)  }^{\frac{p}{2}},
\]
we have%
\begin{equation}
\left \Vert u\right \Vert _{L^{\infty}\left(  \widetilde{Q}_{\sigma}%
,\overline{\nu}\right)  }\leq(\widetilde{k}\left(  \tau,\nu,r\right)
4^{\left(  1+\frac{\nu}{2}\right)  }\left \Vert u\right \Vert _{L^{p}\left(
\widetilde{Q},\overline{\nu}\right)  }^{\frac{p}{2}})\frac{1}{\left(
1-\sigma \right)  ^{1+\frac{\nu}{2}}}\left \Vert u\right \Vert _{L^{\infty
}\left(  \widetilde{Q}_{\rho},\overline{\nu}\right)  }^{1-\frac{p}{2}}.
\label{311}%
\end{equation}

Set $\sigma_{0}=\delta_{1}=\frac{\delta}{\delta^{\prime}}\in \left(
0,1\right)  ,$ $\left(  1-\sigma_{i}\right)  =\left(  \frac{3}{4}\right)
^{i}\left(  1-\delta_{1}\right)  .$ Taking $\sigma=\sigma_{i},$ $\rho
=\sigma_{i+1}$ in $\left(  \ref{311}\right)  $, we have%
\begin{equation}
\left \Vert u\right \Vert _{L^{\infty}\left(  \widetilde{Q}_{\sigma_{i}%
},\overline{\nu}\right)  }\leq(\frac{4}{3})^{i\left(  1+\frac{\nu}{2}\right)
}w\left \Vert u\right \Vert _{L^{\infty}\left(  \widetilde{Q}_{\sigma_{i+1}%
},\overline{\nu}\right)  }^{1-\frac{p}{2}} \label{312}%
\end{equation}
for $w=[\widetilde{k}\left(  \tau,\nu,r\right)  4^{\left(  1+\frac{\nu}%
{2}\right)  }\left \Vert u\right \Vert _{L^{p}\left(  \widetilde{Q}%
,\overline{\nu}\right)  }^{\frac{p}{2}}\left(  1-\delta_{1}\right)
^{-1-\frac{\nu}{2}}]$.

By iterating the inequality $\left(  \ref{312}\right)  $, we obtain%
\[
\left \Vert u\right \Vert _{L^{\infty}\left(  \widetilde{Q}_{\delta_{1}%
},\overline{\nu}\right)  }\leq(\frac{4}{3})^{\left(  1+\frac{\nu}{2}\right)
{\displaystyle \sum \limits_{l\in I_{i-1}}}
l\left(  1-\frac{p}{2}\right)  ^{l}}w^{%
{\displaystyle \sum \limits_{l\in I_{i-1}^{^{\prime}}}}
\left(  1-\frac{p}{2}\right)  ^{l}}\left \Vert u\right \Vert _{L^{\infty}\left(
\widetilde{Q}_{\sigma_{i}},\overline{\nu}\right)  }^{1-\frac{p}{2}},
\]
for $i\in%
\mathbb{N}
$. As before, letting $i$ tend to $\infty$, we have%
\[
\left \Vert u\right \Vert _{L^{\infty}\left(  \widetilde{Q}_{\delta_{1}%
},\overline{\nu}\right)  }\leq(\frac{4}{3})^{\frac{\left(  2+\nu \right)
\left(  2-p\right)  }{p^{2}}}\frac{A\left(  \tau,\nu,2\right)  ^{\frac{2}{p}%
}C^{\frac{\nu}{p}}}{r^{\frac{4}{p}}}4^{\frac{\left(  2+\nu \right)  }{p}%
}\left(  1-\delta_{1}\right)  ^{-\frac{2+\nu}{p}}\left \Vert u\right \Vert
_{L^{p}\left(  \widetilde{Q},\overline{\nu}\right)  }.
\]

This completes the proof.
\end{proof}

\end{document}